\documentclass[12pt]{amsart}
\usepackage{hyperref}

\usepackage{amsmath,amsfonts,amssymb}
\usepackage{mathrsfs}
\usepackage{verbatim}
\usepackage{amsthm}
\usepackage{mathrsfs}

\newtheorem{theorem}{Theorem}[section]
 \newtheorem{corollary}[theorem]{Corollary}
 \newtheorem{lemma}[theorem]{Lemma}
 \newtheorem{proposition}[theorem]{Proposition}
 \theoremstyle{definition}
 \newtheorem{definition}[theorem]{Definition}
 \theoremstyle{remark}
 \newtheorem{remark}[theorem]{Remark}
  \newtheorem{ex}[theorem]{Example}
 \numberwithin{equation}{section}

\def \bC {\mathbb C}

\def \bN {\mathbb N}

\def \bR {\mathbb R}

\def \bT {\mathbb T}

\def \bZ {\mathbb Z}

\def \cD {\mathcal D}

\def \cF {\mathcal F}

\def \cH {\mathcal H}

\def \cL {\mathcal L}
\def \cM {\mathcal M}

\def \fg {\mathfrak g}

\def \ft {\mathfrak t}

\def \fz {\mathfrak z}

\def \fU {\mathfrak U}

\def\Gh{{\widehat{G}}}
\def \sL{\mathscr L}
\def\Rep{{{\rm Rep}\,}}
\def \RepG {\text{\rm Rep}(G)}
\def \FundG {\text{\rm Fund}(G)}
\def \Fund {\text{\rm Fund}}

\def \spec {\text{\rm Spec}}
\def \tr {\text{\rm Tr}}
\def \id {\text{\rm I}}
\def\Op{{{\rm Op}}}

\def\supp{{{\rm supp}}}

\begin{document}

\title[Differential structure on $\Gh$]
{Differential structure\\ 
 on the dual of a compact lie group}
\author[V. Fischer]
{V\'eronique Fischer}
\address
{Department of Mathematical Sciences\\
University of Bath\\
Claverton Down\\ 
Bath  BA2 7AY, United Kingdom}
\email{v.c.m.fischer@bath.ac.uk}
\subjclass[2010]{43A22, 43A77, 22E15,12H05}
\keywords{harmonic analysis on compact Lie groups,
differential calculus over a non-commutative algebra,
homogeneous Sobolev spaces}

\date{June 2018}

\maketitle

\begin{abstract}
In this paper we define difference operators and homogeneous Sobolev-type spaces on the dual of a compact Lie group. As an application and to show that this defines a relevant differential structure, we state and prove multiplier theorems of H\"ormander, Mihlin and Marcinkiewicz types together with the sharpness in the Sobolev exponent for the one of H\"ormander type. 
 \end{abstract}

\makeatletter
\renewcommand\l@subsection{\@tocline{2}{0pt}{3pc}{5pc}{}}
\makeatother

\tableofcontents

\section{Introduction}
\label{sec_intro}

The notion of differential structure on e.g. smooth manifolds 
relies on fundamental ideas from geometry, analysis and algebra.
It is motivated by questions coming not only from mathematics but also from mathematical and theoretical physics.
A modern instance of these ideas and motivations is to consider  algebras equipped 
with operations that generalise differential calculus, the most interesting cases being non-commutative \cite{Beggs+Majid}.
The first part of this paper is driven by this general quest and proposes an intrinsic differential structure 
on the dual $\Gh$ of a compact Lie group $G$, namely difference operators and homogeneous Sobolev spaces.
The second part is devoted to show how this structure provides the natural framework from an analytical viewpoint.

This structure  has already been alluded to in the previous paper  of the author \cite{monJFA} where intrinsic difference operators were introduced in order to study the pseudo-differential calculus on the group $G$.
The ideas seem to have appeared in the reverse chronological order of what has happened on $\bR^n$ and on manifolds: there, the standard pseudodifferential theory (on which micro-local and semi-classical analysis relies) emerged in the 60's from the study of singular integrals on $\bR^n$ (Calder\'on-Zygmund theory), especially multiplier problems. We refer to Section \ref{subsec_mult_thm_Rn} for a (personal) historical viewpoint on multiplier problems.

In the second and last part of this paper, we discuss 
Fourier multipliers on compact Lie groups
 with conditions  analogue to Marcinkiewicz, Mihlin and H\"ormander's. 
A multiplier symbol $\sigma$ is now a field over the dual  $\Gh$ of the group $G$;
therefore any two multiplier  symbols (and a fortiori operators) may not necessarily commute.
The first study in this context goes back to 1971
with Coifman and Weiss' monograph \cite{coifman+weiss}
where they  developed  the Calder\'on-Zygmund theory in the setting of spaces of homogeneous types and as an application studied the Fourier multipliers of $SU(2)$, see also \cite{coifman+weiss1,coifman+weiss2}. 
Until the recent investigations of Ruzhansky and Wirth \cite{ruzhansky+wirth}, 
the research into Fourier multipliers on compact Lie groups had been  focused on  central multipliers  \cite{clerc,stanton,strichartz,vretare, weiss, weiss_SU}.
To the author's knowledge, the rest of the literature 
 on the Fourier $L^p$-multiplier problem on Lie groups  `in the H\"ormander sense' is restricted to the motion group
 (Rubin in 1976 \cite{rubin}),
to the Heisenberg group stemming from the work of de Michele and Mauceri in 1979 \cite{dM+M}, and to the graded nilpotent Lie groups with   the recent work of the author with Ruzhansky \cite{fischer+ruzhansky}.

Our multiplier theorems will be given in Sections \ref{subsec_multG}
and  \ref{subsec_marcinkiewicz}. More precisely we show a result `of H\"ormander type' and from it, we deduce 
theorems of Mihlin and of Marcinkiewicz types.
The proof is classical and relies on the Calder\'on-Zygmund theory adapted 
to the setting of spaces of homogeneous type as in \cite{coifman+weiss}.
Our approach uses the differential structures on the unitary dual $\Gh$ of the group developed in the first part of the paper, 
in particular homogeneous Sobolev spaces $\dot H^s(\Gh)$.
We reformulate the H\"ormander condition using these 
and show that it is  indeed a hypothesis for an $L^p$-multiplier theorem, 
see Sections \ref{subsec_Sigmas}  and \ref{subsec_multG} respectively.
Furthermore this hypothesis is sharp in the exponent $s$ of the Sobolev space, the critical value being half the topological dimension of the group.

The ideas developed in this paper can be of help in classical problems of harmonic analysis on Lie groups:  for example,
Marcinkiewicz conditions for products of groups
(to obtain generalisation of Marcinkiewicz' \cite{marc} $L^p$-multiplier result for double trigonometric series - still an open problem as formulated in \cite[p.63]{stein_topics}), 
case of homogeneous domains such as the real and complex spheres, 
applications to the study of sub-laplacians, 
$L^p-L^q$-boundedness of Fourier symbols under a Lizorkin-type condition,
estimates of the $L^p-L^q$-bounds of particular operators such as spectral projectors.
Furthermore, the objects developed in this paper may be of interest in modern non-commutative harmonic analysis (e.g. questions studied in \cite{junge+mei+parcet}),
and also in areas where mathematics and theoretical physics intersect such as quantum Riemannian geometry (see e.g. \cite{Beggs+Majid}).

\smallskip

The paper is organised as follows.
In Section \ref{sec_alg}, we define the differential structures on algebras of fields over $\Gh$.
In Section \ref{sec_Fmult}, we recall the interpretation of fields over $\Gh$ as  Fourier multipliers.
In Section \ref{sec_dotHsGh}, we define the homogeneous Sobolev spaces on $\Gh$, 
which allow us to state and prove our multiplier theorems
in Section \ref{sec_mult_thm}.
The function of the Laplace-Beltrami operator will be studied
in Section \ref{sec_f(L)indotHs}, 
and used to state our H\"ormander condition and prove its sharpness.

\medskip

\noindent\textbf{Notation:}
$\bN_0=\{0,1,2,\ldots\}$ denotes the set of non-negative integers
and 
$\bN=\{1,2,\ldots\}$ the set of positive integers.
If $\cH_1$ and $\cH_2$ are two Hilbert spaces, we denote by $\sL(\cH_1,\cH_2)$ the Banach space of the bounded operators from $\cH_1$ to $\cH_2$. If $\cH_1=\cH_2=\cH$ then we write $\sL(\cH_1,\cH_2)=\sL(\cH)$.
We may allow ourselves to write $A\asymp B$ when the quantity $A$ and $B$ are equivalent in the sense that there exists a constant such that $C^{-1}A\leq B\leq C A$.

\section{Algebraic structures on $\Gh$}
\label{sec_alg}

In this section
we develop the differential structure on the unitary dual of a (connected) compact Lie group.
After some necessary preliminaries in Section \ref{subsec_preliminaries}, 
we define the difference operators in Section \ref{subsec_my_diff_op}, 
and the algebras on which it is defined in Section \ref{subsec_Sigma}.
We give a precise definition of the fundamental difference operators in Section \ref{subsec_fund}
and  we obtain a structure of differential algebra on $\Gh$ in Section \ref{subsec_structure}.
In the last section, we study the annihilators of the difference operators. 

\subsection{Preliminaries}
\label{subsec_preliminaries}

In this paper, $G$ always denotes a connected compact Lie group 
and $n$ is its dimension.
Its Lie algebra $\fg$ is the tangent space of $G$ at 
the neutral element $e_G$.
We refer to classical textbooks, e.g.
 \cite{knapp_bk} and \cite{hall} for pedagogical presentationa and proofs of their properties.

The complexification of the real Lie algebra $\fg$ is the complex Lie algebra $\fg_\bC:=\bC\otimes \fg$.
 We denote by $\fU(\fg_\bC)$ its universal enveloping algebra. 
Having fixed a basis $(X_j)_{j=1}^n$ for $\fg$, 
a basis of $\fU(\fg_\bC)$ is given by
 $$
X^\alpha:= X_1^{\alpha_1}\ldots X_n^{\alpha_n},
\quad \alpha=(\alpha_1,\ldots, \alpha_n)\in \bN_0^n,
$$
with the usual convention $X^0=\id$.
For $d\in \bN_0$, 
we also denote by $\fU_d(\fg_\bC)$ the subspace of $\fU(\fg_\bC)$
of elements of degree $\leq d$, 
that is, the vector space of the complex linear combination of  $X^\alpha$, 
$|\alpha|\leq d$. 

A representation of $G$ is any continuous group homomorphism $\pi$ from $G$ to the set of automorphisms of a finite dimensional complex space. 
We will denote this space $\cH_\pi$.
Note that continuity implies smoothness.
If $\pi$ is a representation of the group $G$, then 
we keep the same notation $\pi$ for the corresponding infinitesimal representation of $\fg$ and of $\fU(\fg_\bC)$
given via
\begin{equation}
	\label{eq_pi(X)}
\pi(X)=\partial_{t=0} \pi(\exp_G(t X)).
\end{equation}

Two representations $\pi_1$ and $\pi_2$ of $G$ are equivalent when there exists a map $U:\cH_{\pi_1}\to \cH_{\pi_2}$ intertwining the representations, that is,
such that $\pi_2U =U\pi_1$.  
A representation of $G$ is irreducible when the only sub-spaces invariant under $G$ are trivial.
The dual of the group $G$, denoted by $\Gh$, is the set of  irreducible representations of $G$
modulo equivalence.
We also consider the set $\RepG$ of the equivalence classes of  representations  modulo  equivalence.
When possible, we will often identify a representation of $G$ and its class in $\Gh$ or $\RepG$. 
As a topological set, $\Gh$  is discrete and countable.

The tensor product of two representations $\pi_1,\pi_2\in \RepG$ is denoted as $\pi_1\otimes \pi_2\in \RepG$.
For  $\pi\in \RepG$ and $n\in \bN$, we write $\pi^{\otimes n} = \pi\otimes \ldots\otimes \pi\in \RepG$ for the tensor product of $n$ copies of $\pi$. We adopt the convention that $\pi^{\otimes 0} = 1_\Gh$ is the trivial representation.

\subsection{The algebra $\Sigma$}
\label{subsec_Sigma}

We denote by $\Sigma=\Sigma(G)$ the set of fields of operators over 
$\Gh$ on the Hilbert space $\oplus_{\pi\in \Gh} \cH_\pi$:
$$
\Sigma=\Sigma(G)=\left\{\sigma = \{\sigma(\pi) \in \sL(\cH_\pi) : \pi\in \Gh\}\right\}.
$$
For a general definition of fields of operators, see e.g. \cite[Part II Ch 2]{dixmierVnA}.
One checks easily that $\Sigma$ is an algebra (over $\bC$) for the product of linear mappings.

Since any representation $\pi\in \RepG$ may be written as a finite direct sum 
$\pi= \oplus_j \tau_j$ 
of $\tau_j \in \Gh$, 
any field of operator over $\oplus_{\pi\in \Gh} \cH_\pi$ may be naturally extended over $\oplus_{\pi\in \RepG} \cH_\pi$ 
via 
$\sigma(\pi) := \oplus_j \sigma(\tau_j)$.
We will often identify an element $\sigma\in \Sigma$ with its natural extension as a collection over $\RepG$:
\begin{equation}
\label{eq_identification_sigma_Gh_RepG}
\sigma=\{\sigma(\pi), \pi\in \Gh\}
\ \sim  \
\sigma=\{\sigma(\pi), \pi\in \RepG\}.
\end{equation}

Let us recall that the elements of $\fU(\fg_\bC)$ also provide elements of $\Sigma$:

\begin{ex}
\label{ex_PiX_sigma}
If $X\in \fg$, then $\cF_G(X):= \{\pi(X) : \pi \in \Gh\} \in \Sigma$.
In this case again, the extension to a collection over $\RepG$ is straightforward as it coincides with the definition of $\pi(X)$ for $\pi\in \RepG$, see \eqref{eq_pi(X)}. 

More generally, for any $Y\in \fU(\fg_\bC)$,   
$\cF_G(Y):= \{\pi(Y) : \pi \in \Gh\} \in \Sigma$. Furthermore, 
$\cF_G(Y_1Y_2)=\cF_G(Y_1Y_2)$ for any $Y_1,Y_2\in \fU(\fg_\bC)$.
\end{ex}

In Section \ref{subsec_PWthm}, we will see more generally that the group Fourier transform $\cF_G(f)=\widehat f$
of a distribution $f\in \cD'(G)$ defines an element of $\Sigma$.

\subsection{Difference operators}
\label{subsec_my_diff_op}

Here we recall and explain further the intrinsic definition of difference operators 
given in  \cite[Section 3.2]{monJFA}. 

For any representation $\tau,\pi\in \RepG$ and $\sigma\in\Sigma(G)$,
we define  the linear mapping $\Delta_\tau \sigma (\pi)$ on $\cH_\tau\otimes \cH_\pi$ via:
\begin{equation}
\label{eq_def_Deltatau}
\Delta_\tau \sigma (\pi) := \sigma(\tau\otimes \pi) - 
\sigma(\id_\tau\otimes \pi).
\end{equation}
Hence $\Delta_\tau \sigma$ is a field over $\Gh$ 
of operators on the Hilbert space $\oplus_{\pi\in \Gh} \cH_{\pi}\otimes \cH_\tau$;
note that this Hilbert space may be described as 
$$
\oplus_{\pi\in \Gh} \cH_{\pi}\otimes \cH_\tau 
\simeq L^2(G) \otimes \cH_\tau, 
$$
or equivalently, the Hilbert space of square-integrable $\cH_\tau$-valued functions on $G$. 

\begin{ex}
\label{ex_diff_op_torus}
The dual of  the torus $\bT =\bR/2\pi\bZ$,
is  $\widehat \bT=\{ e_\ell, \ell\in\bZ\}$ 
where $e_\ell(x)=e^{i \ell x}$, $x\in \bT$.
Since $e_\ell \otimes e_m = e_{\ell+m}$, we see:
 $$
 \Delta_{e_\ell} \sigma(e_m) = \sigma(e_{\ell+m}) - \sigma(e_m).
 $$
 Consequently, identifying $\widehat \bT$ with the lattice $\bZ$, 
 the difference operators are the usual
discrete (forward or backward) difference operator.
\end{ex}

\begin{ex}
\label{ex_Diffsigma=1}
Let us consider $\id:=\{\id_\pi, \pi\in \Gh\}\in \Sigma$.
Then $\Delta_\tau \id=0$ for any $\tau\in \RepG$.
\end{ex}

For any $\sigma\in \Sigma$ and $\tau\in \Rep G$,
$\Delta_\tau \sigma$ is in 
$\Sigma_\tau$ which is the set defined as follows:

\begin{definition}
For any $\tau\in \Rep G$,
 $\Sigma_\tau$ denotes the set of fields over $\Gh$ on the Hilbert space $\oplus_{\pi\in \Gh} \cH_\pi\otimes \cH_\tau$:
 $$
\Sigma_\tau
=\left\{\omega=\{\omega(\pi)\in \cH_{\pi}\otimes \cH_\tau : \pi\in \Gh\}\right\}.
$$ 
\end{definition}

One checks easily that $\Sigma_\tau$ is an algebra (over $\bC$) for the product of linear mappings.
Example \ref{ex_diff_op_torus} motivates us to observe that if $\tau$ is of dimension 1, then $\Sigma_\tau = \Sigma$.

As for $\Sigma$, we extend by linearity
a field $\sigma\in\Sigma_\tau$ to a field over $\RepG$, see \eqref{eq_identification_sigma_Gh_RepG}
and the preceding paragraph.
We now  observe that Definition \ref{eq_def_Deltatau} also makes sense for fields $\sigma \in \Sigma_{\varphi}$ for some $\varphi\in \Rep G$. Note that the resulting object is then valued in 
$\Sigma_{\varphi\otimes \tau}$. We have obtained the following definition:
\begin{definition}
For any $\tau,\varphi\in \RepG$,
the difference operator $\Delta_\tau:\Sigma_{\varphi} \to \Sigma_{\varphi\otimes \tau} $ is the operator defined via \eqref{eq_def_Deltatau}.	
\end{definition}

 The  difference operators $\Delta_\tau$
 satisfy the following:
 \begin{lemma}[Leibniz property]
For any $\sigma_1,\sigma_2\in \Sigma_\varphi(G)$ and $\varphi,\tau,\pi\in \RepG$, we have:
 \begin{equation}
\label{eq_Leibniz}
\Delta_\tau (\sigma_1\sigma_2)(\pi)=
\Delta_\tau (\sigma_1) (\pi)\ \sigma_2(\id_\tau\otimes\pi)
+
\sigma_1 (\tau\otimes\pi)\ \Delta_\tau (\sigma_2)(\pi),
\end{equation}
 \end{lemma}
The proof is identical to the one given in  \cite[Section 4.1]{monJFA}:
\begin{proof}
We have:	
\begin{align*}
&\Delta_\tau (\sigma_1\sigma_2)(\pi)
=
(\sigma_1\sigma_2)(\tau\otimes\pi)	
-
(\sigma_1\sigma_2)(\id_\tau\otimes\pi)\\
&\quad=
\left(\sigma_1(\tau\otimes\pi) - \sigma_1(\id_\tau\otimes\pi)\right)
\sigma_2
(\tau\otimes\pi)	
+
\sigma_1(\id_\tau\otimes\pi)
\left(\sigma_2(\tau\otimes\pi)-
\sigma_2(\id_\tau\otimes\pi)\right).	
\end{align*}
Thus \eqref{eq_Leibniz} follows.
\end{proof}

Our definition of difference operators  allows us to compose difference operators.
Naturally, up to the order of tensors, difference operators commute.

\begin{ex}
\label{ex_Diffsigma=piX}
If  $\tau,\tau_1,\pi\in \RepG$ and $X\in \fg_\bC$, 
we have
$$
\Delta_{\tau}\Delta_{\tau_1} \cF_G(X) =0.
$$
 \end{ex}

\begin{proof}
We compute easily
$$
 (\tau_1\otimes\pi)(X)
 =
 \partial_{t=0} (\tau_1\otimes\pi)(e^{tX})
 =
 \tau_1(X)\otimes\id_\pi
 +
 \id_{\tau_1}\otimes\pi(X).
 $$
Hence we obtain readily
\begin{equation}
\label{eq_taupiX}
\forall \tau_1,\pi\in \RepG\qquad
\forall X\in \fg_\bC\qquad
 (\tau_1\otimes\pi)(X)
 =
 \tau_1(X)\otimes\id_\pi
 +
 \id_{\tau_1}\otimes\pi(X).
\end{equation}
 Therefore for any $X\in \fg_\bC$ and $\tau_1,\pi\in \RepG$
\begin{align*}
 \Delta_{\tau_1} \cF_G(X) (\pi)
 &=
 \cF_G(X) (\tau_1\otimes\pi)
 -
 \cF_G(X) (\id_{\tau_1}\otimes\pi)
 =
 (\tau_1\otimes\pi)(X) - (\id_{\tau_1}\otimes\pi)(X) 
\\ &=
 \tau_1(X) \otimes \id_\pi.
\end{align*}
Since $\Delta_{\tau_1} \cF_G(X) (\pi)$ is constant in $\pi$, the result follows.
\end{proof}

Using Examples \ref{ex_Diffsigma=1} and \ref{ex_Diffsigma=piX} 
with the Leibniz property \eqref{eq_Leibniz} recursively, 
we obtain:
\begin{ex}
\label{ex_Diffsigma=piXbeta}
For any $\beta\in \bN_0^n$, 
any $\tau_1,\ldots, \tau_m\in \RepG$ with $m>|\beta|$, and any $\pi\in \RepG$,
$$
\Delta_{\tau_1}\ldots\Delta_{\tau_m}
\cF_G(X^\beta)=0.
$$
Hence $\Delta_{\tau_1}\ldots\Delta_{\tau_m} \cF_G(Y)=0$ for any $Y\in \fU_{m-1} (\fg_\bC)$.
 \end{ex}

\subsection{Fundamental difference operators}
\label{subsec_fund}

The connected compact Lie group $G$  admits a finite set of fundamental representations:
$$
\FundG \subset \Gh \subset \RepG,
$$
in the sense that any representation in $\Gh$ will occur in a tensor products of representations in $\FundG$.
See below for a constructive proof of this via the highest weight theory.

\begin{definition}[\cite{monJFA}]
The \emph{fundamental difference operators} 
are the difference operators 
$$
\Delta_\varphi, \quad \varphi\in \FundG, 
$$
corresponding to fundamental representations.
\end{definition}

We will denote by $\varphi_1,\ldots, \varphi_f$ the fundamental representations of $G$:
$$
\FundG=\{\varphi_1,\ldots, \varphi_f\}.
$$

For $\alpha\in \bN_0^f$, we will write:
$$
\varphi^{\otimes \alpha} := \varphi_1^{\otimes \alpha_1} \otimes \ldots\otimes \varphi_f^{\otimes \alpha_f}
\qquad\mbox{and}\qquad \Delta^\alpha=\Delta_{\varphi_1}^{\alpha_1}\ldots \Delta_{\varphi_f}^{\alpha_f}.
$$

\medskip

\subsubsection*{Construction of the fundamental representations}
If the compact connected Lie group $G$ is semisimple and simply connected, then its fundamental representations are the ones  whose highest weights are the fundamental weights of its Lie algebra.
This can be generalised to any compact (connected) Lie group $G$ without further hypotheses in the following way.
For background material, we refer to \cite{knapp_bk} (especially Chapters IV and V) and \cite{hall} (especially Section 12).
We will also need the following well-known properties and conventions about lattices:
\begin{lemma}
Let $V$ ba a real vector space of finite dimension $n$.
\begin{enumerate}
\item
A lattice $\Gamma$  of $V$ is a discrete subgroup of $(V,+)$ or equivalently a set of the form 
$\oplus_{j=1}^n \bZ e_j$ for a basis $(e_j)_{j=1}^n$ of $V$.
We may call $(e_j)_{j=1}^n$ a basis for $\Gamma$.
\item 
If $\Gamma$ is a lattice of $V$, then 
$$
\Gamma^*:=\{\lambda \in V^* \ : \  \lambda(v)\in \bZ \ \mbox{for all}\ v\in \Gamma \}
$$
is a lattice of the dual $V^*$ of $V$; we may call $\Gamma^*$ the dual lattice.
\end{enumerate}
\end{lemma}

Let us set some notation and convention.
Firstly, we consider a compact connected Lie group $G$ of dimension $n$ and centre $Z$.
We denote by $\fg$ its Lie algebra and by $\fz$ the centre of its Lie algebra which is also the Lie algebra of $Z$.
The derived algebra $\fg_{ss}:=[\fg,\fg]$ of $\fg$ is semi-simple 
 and we have 
 $\fg = \fz \oplus \fg_{ss}$.
 We denote by $G_{ss}=\exp \fg_{ss}$ the closed connected subgroup of $G$ whose Lie algebra is $\fg_{ss}$.
Its centre $Z_{ss}:=Z\cap G_{ss}:=\{z_\ell,\ell=1 ,\ldots, |Z_{ss}|\}$ of $G_{ss}$ is finite. 
The centre of $G$ can be written as the direct product $Z=(Z)_0 \times Z_{ss} $ where $(Z)_0=\exp \fz$ is the connected component of $Z$ at the neutral element.

Secondly, we choose a maximal torus $T$ of $G$.
Its Lie algebra $\ft$ decomposes as $\ft= \ft_\fz \oplus \ft_{ss}$
where $\ft_{\fz}:=\ft \cap \fz$ and $\ft_{ss}:=\ft \cap \fg_{ss}$.
The kernel of the homomorphism  $\exp_G :\ft \to T$ between abelian groups is a lattice $\Gamma$  of  $\ft$.
In general,  the lattices $\Gamma\cap \ft_{\fz}$ and $\Gamma \cap \ft_{ss}$  are different from the projections of $\Gamma$ onto $\ft_{\fz}$ and $\ft_{ss}$.
Indeed $Z_{ss} \subset T\cap G_{ss}$,
so for each $\ell=1,\ldots,|Z_{ss}|$, there exists  $H_{ss,\ell}\in \ft_{ss}$, $H_{\fz,\ell}\in \ft_{\fz}$ such that 
$z_\ell=\exp H_{ss,\ell}=\exp H_{\fz,\ell}$.
Consequently $|Z_{ss}|H_{ss,\ell}\in \Gamma_{ss}$ and $|Z_{ss}|H_{\fz,\ell} \in \Gamma_{\fz}$, and we have
$$
\Gamma = \sum_{\ell=1}^{|Z_{ss}|}\bZ( H_{ss,\ell} - H_{\fz,\ell}) + \Gamma_{ss}   +\Gamma_{\fz}.
$$

Thirdly, the set of algebraic elements on $\ft$ is 
$\Lambda_0:=\ft_\fz ^* \bigoplus \oplus_{j=1}^{d_{\ft_{ss}}} \bZ \omega_{ss,j}$
where  $(\omega_{ss,j})_{j=1}^{d_{\ft_{ss}}}$ denotes the fundamental weights of the  semisimple Lie algebra $\fg_{ss}$ with respect to its Cartan subalgebra $\ft_{ss}$
having  fixed an ordering on $\ft_{ss}^*$.
The analytical  integral elements on $\ft$ form the  lattice $\Lambda := \Lambda_0 \cap \Gamma ^*$ in $\ft^*$.
We denote by $(\beta_\ell)_{\ell=1}^n$ a basis  of $\Lambda$.
This induces an ordering of $\ft^*$ but also of $\ft_{ss}^*$ and of $\ft^*_{\fz}$.
We may assume that the two orderings on $\ft_{ss}^*$ coincide, so that the $\beta_\ell$'s are dominant.
Let $(\omega_{\fz,k})_{k=1}^{d_{\ft_\fz}}$ be a basis of $(\Gamma\cap \fz)^*$; we may assume that the $\omega_{\fz,k}$'s are positive for the ordering.
Although each  $\omega_{\fz,k}$ may not be in $\Gamma^*$, 
we check easily that 
 $|Z_{ss}|\omega_{\fz,k}$ is in $\Gamma^*$ and in fact in $\Lambda$.
 
 Finally, the theorem of the highest weight for compact connected Lie groups states that
 an irreducible representation on $G$ is characterised up to equivalence by its highest weight which is an analytical integral dominant functional on $\ft$,
 and that any analytical dominant functional $\lambda$ is the highest weight for an irreducible representation $\pi_\lambda$ on $G$. 
 For instance, the linear functional $\tilde \omega_{\fz} := |Z_{ss}| \sum_{k=1}^{d_{\ft_\fz}} 
 \omega_{\fz,k}$ is analytical  dominant; as it is trivial on $\ft_{ss}$, for any $N\in \bZ$, the functional $N\tilde \omega_{\fz}$ is also analytical dominant
 and its corresponding  representation $\pi_{N\tilde \omega_{\fz}}$ is one-dimensional.
 If $\lambda$ is a dominant analytical integral functional, 
 then, for an integer $N\in \bN_0$ large enough, the functional $\tilde \lambda := \lambda +N\tilde \omega_{\fz}$
is positive for the ordering and therefore it is written as a linear combination over $\bN_0$ of $\beta_\ell$, i.e.  $\tilde \lambda = \sum_{\ell=1}^n n_\ell \beta_\ell$ with $n_\ell\in \bN_0$.
Since $\tilde \lambda$ is also dominant analytical
with 
$\pi_{\tilde \lambda} =\pi_{N \tilde \omega_{\fz}}\otimes 
\pi_\lambda$, we have
$\pi_\lambda =\pi_{-N \tilde \omega_{\fz}}\otimes \pi_{\tilde \lambda}$.
Note that  $\pi_\lambda$ occurs in $\pi_{-\tilde \omega_{\fz}}^{\otimes N}\otimes\otimes_{\ell=1}^n \pi_{\beta_\ell}^{\otimes n_\ell}$.
Hence we can choose  $\pi_{\beta_1},\ldots, \pi_{\beta_n}, \bar\pi_{-\tilde \omega_{\fz}}$ as a set of fundamental representations.

 \subsection{The structure of differential algebra}
 \label{subsec_structure}

In this section, we define a structure of differential algebra
for the space of finite linear combinations of elements in $A_a:=\oplus_{|\alpha|=a} \Sigma_{\varphi^{\otimes \alpha}}$
$$
A:=\sum_{a=0}^{+\infty} A_a=
{\sum_{\alpha\in \bN_0^f}}^{\oplus} 
\Sigma_{\varphi^{\otimes \alpha}}.
$$
The construction is of interest of its own, although we will not use it per se in this paper.

As already noted, for each $\tau\in \RepG$, 
$\Sigma_\tau$ is an algebra over $\bC$.
Furthermore, for each $\tau,\varphi\in \RepG$, 
we can define the left and right action of $\Sigma_\varphi$ on $\Sigma_{\tau\otimes \varphi}$ via
$$
\left(\sigma_{\varphi}\cdot \sigma_{\tau\otimes \varphi}\right)(\pi)
=
\left(\id_\tau\otimes \sigma_{\varphi}(\pi) \right)\sigma_{\tau\otimes \varphi}(\pi)
\quad\mbox{and}\quad
\left(\sigma_{\tau\otimes \varphi}
\cdot \sigma_{\varphi}\right)(\pi)
=
\sigma_{\tau\otimes \varphi}(\pi) \
 \sigma_{\varphi}(\tau\otimes\pi), 
$$
for $\sigma_{\varphi}\in \Sigma_{\varphi}$, 
$\sigma_{\tau\otimes \varphi}\in \Sigma_{\tau\otimes \varphi}$. 
Hence $\Sigma_{\tau\otimes \varphi}$ is a bi-module over $\Sigma_\varphi$ and 
the Leibniz formula \eqref{eq_Leibniz}
may be reformulated as: 
$$
\Delta_\tau (\sigma_1,\sigma_2) =
(\Delta_\tau \sigma_1)\cdot \sigma_2 + \sigma_1\cdot (\Delta_\tau \sigma_2),
\qquad \sigma_1,\sigma_2\in \Sigma_\varphi.
$$

\medskip

As each $\Sigma_\tau$ is an algebra over $\bC$, 
the space $A$ has a natural structure of algebra over $\bC$.
Furthermore, the actions defined in the paragraph above equip $A$ with a structure of bi-module over $A_0=\Sigma$
as well as of a differential ring with the $\Delta_\varphi$, $\varphi\in \FundG$ as derivations.

 \subsection{Annihilators of fundamental difference operators}
 \label{subsec_annihilator}
  
  Here, we study the elements of $\Sigma$ which are annihilated by difference operators. 
 More precisely, we obtain the following characterisation:
 
  \begin{proposition}
 \label{prop_annilation}
Let $\sigma\in \Sigma$.
\begin{itemize}
\item[Case $s=1$:] The following are equivalent:
\begin{enumerate}
\item 	$\Delta_\varphi \sigma=0$
for any $\varphi \in \FundG$,
\item $\Delta_\varphi \sigma=0$
for any $\varphi \in \RepG$,
\item $\sigma=\sigma(1_{\Gh})\id$.
\end{enumerate}

\item[Case $s\in \bN$:] For any $s\in \bN$, the following are equivalent:
\begin{enumerate}
\item $\Delta^\alpha\sigma =0$
for all $\alpha \in \bN_0^f$, $|\alpha|=s$, 
\item 
$\Delta_{\tau_1}\ldots\Delta_{\tau_s}\sigma =0$
for all $\tau_1,\ldots, \tau_s\in \RepG$
\item 
$\sigma\in \cF_G(\fU_{s-1}(\fg_\bC))$.
\end{enumerate}
\end{itemize}
\end{proposition}

The case $s=1$ is easily proved:

\begin{proof}[Proof of Proposition \ref{prop_annilation}, case $s=1$]
Let $\sigma\in \Sigma$ be such that $\Delta_\varphi \sigma=0$
for any $\varphi \in \FundG$.
Then 
\begin{equation}
\label{eq_pf_lem_diff_I}
\forall \pi\in \Gh, \quad \forall \varphi\in \FundG\qquad
\sigma(\varphi\otimes \pi)=\sigma(\id_\varphi\otimes \pi).
\end{equation}
Applying \eqref{eq_pf_lem_diff_I} to $\pi=1_{\Gh}$, 
we obtain $\sigma(\varphi)=\sigma(1_{\Gh})\id_\varphi$ for  any $\varphi \in \FundG$.
Using this and applying  \eqref{eq_pf_lem_diff_I}, 
we obtain $\sigma(\varphi\otimes \varphi' )=\sigma(1_{\Gh})\id_{\varphi\otimes \varphi'}$ for  any $\varphi,\varphi' \in \FundG$.
Recursively we obtain that $\sigma$ must be equal to the identity up to the constant $\sigma(1_{\Gh})$ at any representation $\pi$ which can be written as a tensor product of fundamental representations, 
thus by linearity  for any $\pi \in \RepG$.
Hence $\sigma=\sigma(1_{\Gh})\id$.
This shows that (1)$\Rightarrow$(3).

By Example \ref{ex_Diffsigma=1}, 
(3)$\Rightarrow$(2).
We also have 
(2)$\Rightarrow$(1).
This shows the case $s=1$ in Proposition \ref{prop_annilation}.
\end{proof}

We now turn our attention to second-order fundamental difference operators:  

 \begin{lemma}
\label{lem_diff2}
Let $\sigma\in \Sigma$ be such that $\sigma(1_{\Gh})=0$.
The following properties are equivalent:
\begin{enumerate}
\item
 \label{item_lem_diff2_sigma=}
$\sigma \in\cF_G(\fg_\bC)$, 
i.e. there exists   $X\in \fg_\bC$ such that 
$\sigma=\cF_G(X)=\{\pi(X), \pi\in \Gh\}$,
\item 
 \label{item_lem_diff2_leibRep}
for all $\pi_1,\pi_2\in \RepG$,
we have
\begin{equation}
\label{eq_lem_diff2_leibRep}
\sigma(\pi_1\otimes\pi_2)
=
\sigma(\pi_1)\otimes \id_{\pi_2}
+
\id_{\pi_1}\otimes \sigma(\pi_2),
\end{equation}
\item 
 \label{item_lem_diff2_leibRep+Fund}
 the same as \eqref{item_lem_diff2_leibRep}
with  $\pi_1\in \FundG$ and $\pi_2\in \RepG$,
\item 
 \label{item_lem_diff2_Delta2}
for all $\varphi,\varphi' \in \FundG$, 
$\Delta_{\varphi}\Delta_{\varphi'} \sigma=0$.
\end{enumerate}
\end{lemma}

We check readily the implications
\eqref{item_lem_diff2_sigma=}
$\Rightarrow$
\eqref{item_lem_diff2_leibRep}
$\Rightarrow$
\eqref{item_lem_diff2_leibRep+Fund},
and \eqref{item_lem_diff2_leibRep}$\Rightarrow$
\eqref{item_lem_diff2_Delta2},
see Example \ref{ex_Diffsigma=piX}, 
especially \eqref{eq_taupiX}.

\begin{proof}[Proof of \eqref{item_lem_diff2_leibRep+Fund}$\Rightarrow$\eqref{item_lem_diff2_leibRep}]
Let $\sigma\in \Sigma$ satisfying $\sigma(1_{\Gh})=0$
and Property \eqref{item_lem_diff2_leibRep+Fund}.
Let $\varphi,\varphi'\in \FundG$ and $\pi\in \RepG$.
Property \eqref{item_lem_diff2_leibRep+Fund} implies
\begin{align*}
\sigma(\varphi\otimes\varphi'\otimes \pi)
&=
\sigma(\varphi)\otimes \id_{\varphi'\otimes \pi}
+\id_{\varphi}\sigma(\varphi'\otimes \pi)
\\
&=
\sigma(\varphi)\otimes \id_{\varphi'}\otimes\id_ \pi
+\id_{\varphi}\otimes \sigma(\varphi') \otimes \id_\pi
+\id_{\varphi}\otimes \id_{\varphi'} \otimes \sigma(\pi)
\\
&=
\sigma(\varphi\otimes \varphi')\otimes \id_\pi
+\id_{\varphi\otimes \varphi'} \otimes \sigma(\pi).
\end{align*}
We recognise \eqref{eq_lem_diff2_leibRep} with $\pi_1=\varphi\otimes \varphi'$.
Therefore  \eqref{eq_lem_diff2_leibRep} holds for any $\pi_2\in \RepG$ 
and $\pi_1$ any fundamental representation or tensor of any two fundamental representations.
Proceeding recursively, we see that \eqref{eq_lem_diff2_leibRep} holds
for $\pi_1$ being any tensors of fundamental representations.
Since any representation $\pi_1\in \Gh$ occurs in the 
decomposition into irreducibles of tensor of fundamental representations, 
this shows that \eqref{eq_lem_diff2_leibRep} holds
for any $\pi_1\in \Gh$, and thus for any $\pi_1\in \RepG$.
\end{proof}

\begin{proof}[Proof of \eqref{item_lem_diff2_Delta2}$\Rightarrow$\eqref{item_lem_diff2_leibRep+Fund}]
Let $\sigma\in \Sigma$ satisfying $\sigma(1_{\Gh})=0$
and Property \eqref{item_lem_diff2_Delta2}.
Let $\varphi'\in \FundG$.
Each entry of $\Delta_{\varphi'}\sigma$ is an element of $\Sigma$ annihilated by all the difference operators $\Delta_{\varphi}$, $\varphi\in \FundG$, 
and,by the case $s=1$ which is already proved, is of the form $c \id$ for some complex constant $c$.
Writing with tensors, this means that there exists $M_{\varphi'}\in \sL(\cH_{\varphi'})$ such that
$$
\forall \pi\in \RepG\qquad
\Delta_{\varphi'}\sigma (\pi)=M_{\varphi'} \otimes \id_\pi.
$$
Applying this to $\pi=1_{\Gh}$, we obtain
$$
M_{\varphi'} =\Delta_{\varphi'}\sigma(1_{\Gh})
=\sigma(\varphi'\otimes 1_{\Gh}) -\sigma(\id_\tau\otimes 1_{\Gh})
=\sigma(\varphi'),
$$
since $\sigma(1_{\Gh})=0$.
Therefore $\Delta_{\varphi'}\sigma (\pi)=\sigma(\varphi') \otimes \id_\pi$
and Property \eqref{item_lem_diff2_leibRep+Fund} holds.
\end{proof}

Hence Lemma \ref{lem_diff2} will be proved when we show the implication \eqref{item_lem_diff2_leibRep}$\Rightarrow$\eqref{item_lem_diff2_sigma=}.
Let us postpone the proof of this implication  
until Section \ref{subsec_pf_item_lem_diff2_leibRep}
and show the general case $s\in \bN$ recursively.
We will need the simple computational remark:
\begin{equation}
\label{eq_Delta_tensor}
\forall \sigma\in \Sigma, \ \pi_1,\pi_2,\pi\in \RepG
\quad
\Delta_{\pi_1\otimes \pi_2}\sigma(\pi)
=
\Delta_{\pi_1}\sigma(\pi_2\otimes \pi) + \id_{\pi_1}\otimes \Delta_{\pi_2}\sigma(\pi).
\end{equation}
which follows from the definition of difference operators.

\begin{proof}[Proof of Proposition \ref{prop_annilation}, general case]
First we observe that 
 (2)$\Rightarrow$(1) and,
by Example \ref{ex_Diffsigma=piXbeta}, (3)$\Rightarrow$(2). So it remains to show (1)$\Rightarrow$(3).
Let us prove this recursively on $s$.
The equivalences for $s=1,2$ are already proved (or assumed).
Let us assume that the equivalences have been established for $s=1,2,\ldots,s_0$.

Let $\sigma\in \Sigma$  annihilated by the difference operators 
$\Delta^\alpha$, $|\alpha|=s_0+1$.
Then by \eqref{eq_Delta_tensor}, it is  annihilated by the difference operators 
$\Delta_\tau\Delta^{\alpha}$
for any $\alpha\in \bN_0^f$, $|\alpha|=s_0$,
and $\tau$ being a fundamental representation or a tensor of two fundamental representations, and recursively any tensor of fundamental representations, therefore any non-trivial representation in fact.
Since $\Delta_{1_{\Gh}}$ is zero on $\Sigma$, 
we can take $\tau$ even trivial.
Then for each $\tau\in \RepG$, 
the $\cH_\tau$-entries of $\Delta_\tau \sigma$ is annihilated by any product of $s_0$ fundamental operators
and therefore we have:
\begin{equation}
\label{eq_Deltasigma=+lot}
\Delta_\tau \sigma(\pi) 
=
\sum_{|\beta|=s_0}M_{\tau,\beta} \otimes \pi(X)^\beta
+ l.o.t.
\end{equation}
where $M_{\tau,\beta}\in \sL(\cH_\tau)$
and 
$l.o.t.$ stand for lower order terms, that is here, 
modulo
$\sL(\cH_\tau)\otimes \pi(\fU_{s_0-1}(\fg_\bC))$.
Note that this writing is unique. 
So  for each $\beta\in \bN_0^n$, 
we have obtained  $\{M_{\tau,\beta}, \tau\in \RepG\} \in \Sigma$
which is 0 at $\tau=1_{\Gh}$.

Furthermore, for any $\tau,\tau'\in \RepG$, we have:
\begin{align*}
\Delta_{\tau\otimes\tau'} \sigma(\pi) 
&=
\sum_{|\beta|=s_0}M_{\tau\otimes \tau',\beta} \otimes \pi(X)^\beta
+ l.o.t.,
\\
\Delta_{\tau} \sigma(\tau'\otimes\pi) 
&=
\sum_{|\beta|=s_0}M_{\tau,\beta} \otimes (\tau'\otimes\pi)(X)^\beta
+ l.o.t.
\\
&=
\sum_{|\beta|=s_0}M_{\tau,\beta} \otimes \id_{\tau'}\otimes\pi(X)^\beta
+ l.o.t.,
\end{align*}
 so that \eqref{eq_Delta_tensor} implies that 
$$
\sum_{|\beta|=s_0}M_{\tau\otimes \tau',\beta} \otimes \pi(X)^\beta
=
\sum_{|\beta|=s_0}(M_{\tau,\beta} \otimes \id_{\tau'}
+\id_{\tau}\otimes M_{\tau',\beta})
\otimes\pi(X)^\beta.
$$
Hence each field $\{M_{\tau,\beta}, \tau\in \RepG\}$
satisfies Property \eqref{item_lem_diff2_leibRep} of Lemma \ref{lem_diff2}.
It has to be of the form $\cF_G(Y_\beta)$ for some $Y_\beta\in \fg_\bC$.
Because of \eqref{eq_Deltasigma=+lot},
$\sigma_1:=\sigma -\sum_{|\beta|=s_0}\cF_G(Y_\beta X^\beta) \in \Sigma$
is such that $\Delta_\tau \sigma$
is of the form 
\eqref{eq_Deltasigma=+lot}
with $s_0$ replaced by $s_0-1$.
Proceeding as above, we obtain recursively on $s_0$ that $\sigma\in \cF_G(\fU_{s_0-1}(\fg_\bC))$.
This concludes the proof of  Proposition \ref{prop_annilation}.
\end{proof}

Proposition \ref{prop_annilation} and Lemma \ref{lem_diff2} 
are proved once we show the implication \eqref{item_lem_diff2_leibRep}$\Rightarrow$\eqref{item_lem_diff2_sigma=}
in Lemma \ref{lem_diff2}.
This will be done in Section \ref{subsec_pf_item_lem_diff2_leibRep}.
Our proof will use the interpretation of $\Sigma$ as Fourier multipliers which we now present.

\section{Interpretation as Fourier multipliers}
\label{sec_Fmult}

Here we recall  the well-known viewpoint on the space $\Sigma$ as the space of symbols for Fourier multiplier operators.
This will be an opportunity to set notation and conventions, especially regarding differential operators such as the Laplace-Beltrami operator.
References include \cite{stein_topics,larsen,clerc,coifman+weiss1,coifman+weiss2,coifman+weiss,dM+M,stanton,strichartz,vretare,weiss_SU}.
Eventually, in Section \ref{subsec_pf_item_lem_diff2_leibRep}, 
we will present
 the missing part in the proof of Lemma \ref{lem_diff2}.

\subsection{The Peter-Weyl theorem}
\label{subsec_PWthm}

For any representation $\pi$ of $G$, one can always find an inner product (often denoted by $(\cdot,\cdot)_{\cH_\pi}$) on $\cH_\pi$ such that the map $\pi(g)$ is unitary on $\cH_\pi$.
If $\pi$ is a representation of $G$,
then its coefficients are all the functions of the form 
$$
(\pi \ u,v): x\mapsto (\pi(x)u,v)_{\cH_\pi},\qquad u,v\in \cH_\pi.
$$
These are smooth functions on $G$.
If a basis $\{e_1,\ldots, e_{d_\pi}\}$ of $\cH_\pi$ is fixed, 
then the matrix coefficients of $\pi$
are the coefficients $\pi_{i,j}=(\pi \ e_i,e_j)_{\cH_\pi}$, $1\leq i,j\leq d_\pi$.
The cofficients of $\pi$
form the finite-dimensional complex vector space
$$
L^2_\pi(G) := \{(\pi \ u,v), u,v\in \cH_\pi\}.
$$
Note that $L^2_\pi(G)$ depends only on the equivalence class of $\pi$.

In this paper, we assume that a Haar measure on $G$ has been chosen and that it is normalised to be a probability measure.
As is customary, for $p\in [1,\infty)$, 
we denote by $L^p(G)$ the usual Banach space of measurable function $f:G\rightarrow \bC$ such that $|f|^p$ is integrable.

\begin{theorem}[Peter-Weyl Theorem]
\label{thm_PW}
The dual $\Gh$ is discrete and countable.
The Hilbert space $L^2(G)$ decomposes as the Hilbert direct sum $\oplus_{\pi\in \Gh} L^2_\pi(G)$.
Moreover, if for each $\pi\in \Gh$, one fixes a realisation as a representation with an orthonormal basis of $\cH_\pi$, 
then the functions $\sqrt d_\pi \pi_{i,j}$, $1\leq i,j \leq d_\pi$, $\pi\in \Gh$, form an orthonormal basis of $G$.
\end{theorem}

If $f\in \cD'(G)$ is a distribution and $\pi$ is a unitary representation, 
 its group Fourier transform at $\pi$
is denoted by
$$
\pi(f)\equiv \widehat f(\pi)\equiv \cF_G f(\pi) \in \sL(\cH_\pi)
$$
and defined  via
$$
 \pi(f)  = \int_G f(x)\pi(x)^* dx,
\quad \mbox{i.e.}\quad
(\pi(f)u,v)_{\cH_\pi} = \int_G f(x)(u,\pi(x)v)_{\cH_\pi} dx,
$$
since the coefficient functions are smooth.
In particular, we consider the Fourier transform of a function to be defined on $\RepG$ and by restriction on $\Gh$.
Note $\widehat f \in \Sigma$, see Section \ref{subsec_Sigma}.

If $f$ is integrable and $\pi$ unitary, we have 
\begin{equation}
\label{eq_cF_L1}
\|\cF_G \kappa(\pi)\|_{\sL(\cH_\pi)}\leq \|\kappa\|_{L^1(G)}.
\end{equation}

One checks easily that the group Fourier transform maps the convolution  of two distributions $f_1,f_2\in \cD'(G)$
to the matrix product or composition of their group Fourier transforms: 
$$
\cF_G(f_1*f_2) = \widehat f_2\ \widehat f_1.
$$
Recall that the (non-commutative) convolution on $G$ is defined in the usual way:
$$
f*g (x) =\int_G f(y) g(y^{-1}x) dy,
\qquad f,g\in \cD'(G).
$$

The Peter-Weyl theorem yields the \emph{Plancherel formula}:
$$
\int_G|f(x)|^2 dx
=
\sum_{\pi\in \Gh} d_\pi \|\pi(f)\|_{HS(\cH_\pi)}^2,
\quad f\in L^2(G),
$$
and the \emph{Fourier inversion formula}
$$
f(x)
=
\sum_{\pi\in \Gh} d_\pi \tr\left( \pi(x) \pi(f)\right),
\quad f\in \cD(G), \ x\in G.
$$

The finite linear sums of vectors in some $L^2_\pi(G)$, $\pi\in \Gh$ form the vector space:
\begin{equation}
\label{eq_def_L2f}
L^2_{fin}(G) := {\sum_{\pi\in \Gh}}^\perp L^2_\pi(G).
\end{equation}
As each $L^2_\pi(G)$ is a finite dimensional subspace of $\cD(G)$, 
$L^2_{fin}(G)\subset \cD(G)$.
The Peter-Weyl Theorem can be stated equivalently as follows:
$L^2_{fin}(G)$ is dense in $L^2(G)$ and 
\begin{equation}
\label{eq_cq_Peter+Weyl_thm}
d_\pi \widehat \pi (\pi')= \delta_{\pi=\pi'} \id_{\pi'},
\end{equation}
for any two representations $\pi,\pi'\in \Gh$.
Here $\id_{\pi'}$ is the identity operator on $\cH_{\pi'}$
and \eqref{eq_cq_Peter+Weyl_thm} means that, when $\pi$ is realised as a matrix representation, we have 
$d_\pi \widehat \pi_{i,j} (\pi')= 0$ for any $i,j$ satisfying $1\leq i \not= j\leq d_\pi$,
and 
$d_\pi \widehat \pi_{i,1} (\pi')= 1$ for any $i$ satisfying $1\leq i \leq d_\pi$.

\subsection{Fourier multipliers}
\label{subsec_Fmult}

Recall that, on the torus, 
a Fourier multiplier
 corresponding to the symbol $\sigma :\bZ \to \bC$
 is the operator $\Op(\sigma)$ given via
\begin{equation}
\label{eq_Opsigma_T}
\Op(\sigma) \phi (e^{i\theta})
=\sum_{\ell=-\infty}^{+\infty} 
e^{i\ell \theta}
\sigma(\ell) \widehat \phi(\ell),
\end{equation}
where the function $\phi$ admits the Fourier expansion $\phi(e^{i\theta}) = \sum_{\ell=-\infty}^{+\infty} 
e^{i\ell \theta} \widehat \phi(\ell)$.
Clearly $\Op(\sigma)$  is a linear operator well defined 
on the space of functions that have  only a finite number of non-zero Fourier coefficients for instance.
Furthermore the Plancherel formula implies that 
the corresponding multiplier operator  admits a unique extension as a linear operator   bounded on $L^2(\bT)$
if and only if the symbol is a bounded sequence.

All this is easily generalised to the case of an arbitrary compact Lie group in the following way.

\begin{definition}
We denote by $L^2(\Gh:\Sigma)$ the Hilbert space of $\sigma \in \Sigma$ such that
$$   
\|\sigma\|^2_{L^2(\Gh:\Sigma)}:=\sum_{\pi\in \Gh} d_\pi \|\sigma(\pi)\|_{HS(\cH_\pi)}^2
$$
is finite. 
\end{definition}
The Peter-Weyl theorem (Theorem \ref{thm_PW})
implies that the Hilbert space $L^2(\Gh:\Sigma)$ is
 isometrically isomorphic to $L^2(G)$ via the Fourier transform.

\begin{definition}
We denote by $L^2_{fin}(\Gh:\Sigma)$ the space of symbols $\sigma$ such that $\sigma(\pi)=0$ for all $\pi\in \Gh$ but a finite number.
\end{definition}
Naturally $L^2_{fin}(\Gh:\Sigma)$ is the subspace of $L^2(\Gh:\Sigma)$ 
isometric to $L^2_{fin}(G)$ via the Fourier transform, see \eqref{eq_def_L2f},
and $L^2_{fin}(\Gh:\Sigma)$ is a dense  subspace of $L^2(\Gh:\Sigma)$.

\begin{definition}
\label{def_op}
The \emph{Fourier multiplier operator} with symbol $\sigma \in \Sigma$  is the operator $\Op(\sigma)$ defined on $L^2_{fin}(G)$ via
$$
\Op(\sigma) \phi (x)=\sum_{\pi\in \Gh}
d_\pi \tr \left(\pi(x) \sigma(\pi) \widehat \phi(\pi)\right),
\quad
\phi\in L^2_{fin}(G), \ x\in G.
$$
\end{definition}

By definition of $L^2_{fin}(G)$, the summation above is finite
and
this defines a linear  operator $\Op(\sigma):L^2_{fin}(G)\rightarrow L^2_{fin}(G)$ satisfying 
$$
\cF_G\left\{\Op(\sigma) \phi\right\} = \sigma \widehat \phi.
$$
Conversely, 
if  $T$ is a linear operator defined on $L^2_{fin}(G)$
(and with image some complex-valued functions of $x\in G$)
and if $T$ is invariant under left-translations, 
then the symbol is recovered via
\begin{equation}
\label{eq_sigma_T}
\sigma(\pi)=
\pi (x)^* (T \pi)(x), 
\quad \mbox{that is,}\quad
[\sigma(x,\pi)]_{i,j}=
\sum_{k} \overline{\pi_{ki} (x)} (T \pi_{kj})(x) ,
\end{equation}
for any $x\in G$, for instance $x=e_G$;
here one has fixed a matrix realisation of $\pi$ 
but \eqref{eq_sigma_T} is in fact independent of this realisation.
This can be easily checked using \eqref{eq_cq_Peter+Weyl_thm}.
This shows that the quantisation $\Op$ defined above is injective.
Moreover  \eqref{eq_sigma_T} makes sense for any $\pi\in \RepG$
and one checks easily that this coincides with the natural extension of $\sigma$ to a collection over $\RepG$.
In other words, 
the identification \eqref{eq_identification_sigma_Gh_RepG} may be realised via
\begin{equation}
\label{eq_identification_sigma_Gh_RepG1}
\sigma(\pi) := (\Op(\sigma) (\pi))(e_G), 
\quad \pi\in \RepG.
\end{equation}

The considerations above imply
that the space $\Sigma$ stands in bijection with the space of Fourier multipliers.
For this reason, the elements of $\Sigma$ may be called \emph{symbols}.

\medskip

Roughly speaking, any `reasonable' convolution operator may be viewed as a Fourier multiplier:
\begin{ex}
\label{ex_symbol_op}
If $\kappa\in \cD'(G)$, then 
the convolution operator $T_\kappa:\cD(G) \to \cD'(G)$, $T_\kappa (\phi)=\phi*\kappa$
with kernel $\kappa$ is the extension of the  Fourier multiplier
$\Op(\widehat \kappa)$ with symbol $\widehat \kappa$.
Indeed, we have 
$$
\widehat  {T_\kappa\phi} = \widehat \kappa \widehat \phi, \qquad \phi\in \cD(G).
$$
\end{ex}

\noindent\textbf{Convention:}
in this paper, we allow ourselves to keep the same notation for a linear operator $T$ with $T:L^2_{fin}(G)\rightarrow L^2_{fin}(G)$  or $T:\cD(G)\to\cD'(G)$ and any of its possible extension as a continuous operator on a topological  spaces of functions on of $G$
as long as such an extension exists and is unique. 

\medskip

Let us give more concrete examples.

\begin{ex}
\label{ex_cFXbeta}
The identity operator on $\cD(G)$ 
is a Fourier multiplier with kernel  $\delta_{e_G}$
and  symbol  $\widehat \delta_{e_G} =\id$.
More generally
any left-invariant differential operator is a Fourier multiplier
since, 
for any $\beta\in \bN_0^n$,
the operator $X^\beta$ is a Fourier multiplier with kernel  $(X^\beta)^t\delta_{e_G}$
and  symbol 
$$
\cF_G(X)^\beta:=\{\pi(X)^\beta, \pi\in \Gh\}.
$$
\end{ex}
\noindent\textbf{Convention:}
Here, $(X^\beta)^t$ denotes the transpose of $X^\beta$:
$(X^\beta)^t = (-1)^{|\beta|}
X_n^{\beta_n}\ldots X_1^{\beta_1}$.
We have also used the usual identification of  $\fg$ with the space of vector fields on $G$ which are invariant under left translations, 
and consequently of $\fU(\fg_\bC)$ with the Lie algebra of the differential operators on $G$ which are invariant under left translations.

\begin{definition}
We denote by $L^\infty(\Gh:\Sigma)$ the Banach space of $\sigma\in \Sigma$ such that
$$
\|\sigma\|_{L^\infty(\Gh:\Sigma)}
:=
\sup_{\pi\in \Gh} \|\sigma(\pi)\|_{\sL(\cH_\pi)}
=
\sup_{\pi\in \RepG} \|\sigma(\pi)\|_{\sL(\cH_\pi)},
$$
is finite. 
\end{definition}

Recall that
if an operator $T\in \sL(L^2(G))$ is left-invariant, that is, invariant under left-translation: $T(f(x_0\cdot))(x)= (Tf) (x_0x)$, $x,x_0\in G$, $f\in L^2(G)$, 
then the Schwartz kernel theorem implies that 
it is a right convolution operator in the sense that there exists $\kappa\in \cD'(G)$ such that $T=T_\kappa:\phi\mapsto \phi*\kappa$ on $\cD(G)$.
The Peter-Weyl theorem implies that
the Banach space of 
operators which are left-invariant and bounded on $L^2(G)$ is isometric to $L^\infty(\Gh:\Sigma)$.
Indeed,  if  $\sigma\in L^\infty(\Gh:\Sigma)$, then the corresponding Fourier mulitplier $\Op(\sigma)$
is bounded on $L^2(G)$ with operator norm 
\begin{equation}
\label{eq_L2bdd_LinftyGh}
\|\Op(\sigma)\|_{\sL(L^2(G))}=\|\sigma\|_{L^\infty(\Gh:\Sigma)}
\end{equation}
The converse holds easily: if $\Op(\sigma)$
is bounded on $L^2(G)$ then $\|\sigma\|_{L^\infty(\Gh:\Sigma)}$ is finite.
Furthermore, Equation \eqref{eq_L2bdd_LinftyGh} yields
\begin{equation}
\label{eq_Tkappa_sup}
\|T_\kappa\|_{\sL(L^2(G))}
=\sup_{\pi\in \Gh} \|\cF_G \kappa(\pi)\|_{\sL(\cH_\pi)}.
\end{equation}

\subsection{The Laplace-Beltrami operator and its Fourier transform}

The (positive) Laplace-Beltrami operator  is 
$$
\cL:= -X_1^2-\ldots- X_n^2,
$$
where $X_1,\ldots,X_n$ is an orthonormal basis of $\fg$;
here we assume that $\fg$ is equipped with of a scalar product invariant under the adjoint representation of $G$ (this is always possible).

The operator $\cL$ does not depend on a particular choice of such a basis. 
It is invariant under left and right translations
and its group Fourier transform is scalar:
$$
\forall \pi\in \Gh\quad \exists ! \lambda_\pi\in [0,\infty)\quad
\pi(\cL) =  \lambda_\pi \id_\pi.
$$

We keep the same notation for $\cL$ and its self-adjoint extension on $L^2(G)$
having as domain of definition the space of all functions $f\in L^2(G)$ 
such that $\cL f\in L^2(G)$. 
Then $\cL$ is a positive self-adjoint operator on $L^2(G)$.
The Peter-Weyl Theorem
yields an explicit spectral decomposition for $\cL$ and of its spectrum
$\spec(\cL)=\{\lambda_\pi,\pi\in \Gh\}$.

For any function $f:[0,\infty)\to \bC$ 
the spectral multiplier $f(\cL)$ is a well defined linear operator on $L^2_{fin}(G)$ with symbol $f(\widehat \cL):=\{f( \lambda_\pi)\id_\pi, \pi\in \Gh\}$.
If $f(\cL)$ extends to a continuous operator $\cD(G)\to \cD'(G)$, then
by the Schwartz kernel theorem, it admits a distributional convolution kernel which we denote by $f(\cL)\delta_{e_G}\in \cD'(G)$:
$$
f(\cL)\phi = \phi * (f(\cL)\delta_{e_G}),
\quad \phi\in \cD(G).
$$
Naturally the group Fourier transform of this distribution is the symbol of the operator:
$$
\cF(f(\cL)\delta_{e_G})(\pi) = f(\lambda_\pi)\id_\pi, \quad \pi\in \Gh.
$$

The operators $f(\cL)$ and their kernels
have been extensively studied, see e.g. Theorem \ref{thm_kernelf(L)} and its proof for a sample of results and references.
In the context of our argument, we will study these operators in more details 
in Section \ref{sec_f(L)indotHs}.
Furthermore, 
the proof given below of the implication \eqref{item_lem_diff2_leibRep}$\Rightarrow$\eqref{item_lem_diff2_sigma=} in Lemma \ref{lem_diff2}
 will use  some properties of $\lambda_\pi$ 
which come from the very rigid structure of weights and roots of compact Lie groups.

\subsection{Proof of \eqref{item_lem_diff2_leibRep}$\Rightarrow$\eqref{item_lem_diff2_sigma=} in Lemma \ref{lem_diff2}}
\label{subsec_pf_item_lem_diff2_leibRep}

Before starting the proof of the missing implication in Lemma \ref{lem_diff2}, 
let us summarise the properties of the eigenvalues of $\cL$ which will be used in the proof.
We will use the notion of fundamental representations $\varphi_1,\ldots ,\varphi_s$ explained in Section \ref{subsec_fund}
and their corresponding highest weights which we call fundamental weights.

\begin{lemma}
\label{lem_lambdapi}
\begin{enumerate}
\item 
\label{item_lem_lambdapi_eq_omega}
Writing $\omega_\pi $ for 
the highest weight  of a representation $\pi\in \Gh$, 
we have
$ 1+ |\omega_\pi| \asymp (1+\lambda_\pi)^{1/2}$,
in the sense that 
there exists $C>1$ 
such that 
$$
\forall \pi\in \Gh\qquad
C^{-1}(1+ |\omega_\pi|) 
\leq (1+\lambda_\pi)^{1/2}
\leq C(1+ |\omega_\pi|).
$$
\item 
\label{item_lem_lambdapi_eq_s}
The highest weight  of a representation $\pi\in \Gh$ 
can be written as a linear combination
 $\omega_\pi = m_1\omega_1 +\ldots+m_f\omega_f$
 of the fundamental weights  $\omega_1,\ldots,\omega_f$,
 and we have
 $\sum_j m
 _j\asymp 1+ |\omega_\pi|$
in the sense that 
there exists $C>1$ 
such that 
$$
\forall \pi\in \Gh\qquad
C^{-1}(1+ |\omega_\pi|) 
\leq \sum_j m
 _j
\leq C(1+ |\omega_\pi|).
$$
\item
\label{item_lem_lambdapi_eq_lambdarho}

For any $\pi,\rho\in \Gh$, 
if $\rho\in \Gh$ intervenes in the decomposition of $\tau\otimes\pi$ into irreducibles for some $\tau\in \FundG$
then 
$(1+\lambda_\rho) \asymp(1+ \lambda_\pi)$.
Furthermore the supremum 
 over
$\pi,\rho\in \Gh$ and 
 $\tau\in \FundG$ such that
$\rho\subset \tau\otimes\pi$
$$
\sup |\lambda_\rho-\lambda_\pi| <\infty
$$
is finite.
\item
\label{item_lem_lambdapi_eq_lambdarhos}
 Let $s\in \bN$. 
For any $\pi,\rho\in \Gh$, 
if $\rho\in \Gh$ occurs in the decomposition of      $\tau_1\otimes\ldots\otimes \tau_s\otimes\pi$ into irreducibles for some $\tau_1,\ldots,\tau_s\in \FundG$
then 
$(1+\lambda_\rho) \asymp(1+ \lambda_\pi)$.
Furthermore the supremum over
$\pi,\rho\in \Gh$ and 
 $ \tau_1,\ldots,\tau_s\in \FundG$
 such that
 $\rho\subset \tau_1\otimes\ldots\otimes \tau_s\otimes\pi$
$$
\sup |\lambda_\rho-\lambda_\pi| <\infty,
$$
is finite.
\end{enumerate}
\end{lemma}

\begin{proof}[Proof of Lemma \ref{lem_lambdapi}]
The $\cL$-eigenvalue $\lambda_\pi$ on $L^2_\pi(\Gh)$
 can be written in terms of the weight $\omega_\pi$ of $\pi\in \Gh$
 as
$\lambda_\pi=|\omega_\pi + \delta|^2 -|\delta|^2$,
where $\delta$ denotes the half sum of the positive roots;
see \cite[Proposition 5.28]{knapp_bk} for the semi-simple case, 
which extends readily to the general case.
This implies  Part \eqref{item_lem_lambdapi_eq_omega}
and Part \eqref{item_lem_lambdapi_eq_s}
since there are only a finite number of fundamental representations.

Let us prove Part \eqref{item_lem_lambdapi_eq_lambdarho}.
Let $\pi,\rho\in \Gh$ and $\tau\in \FundG$
such that $\rho\in \Gh$ occurs in the decomposition of $\tau\otimes \pi$
into irreducibles.
By \cite[Proposition 9.72]{knapp_bk}, $\omega_\rho = \omega_\pi + \mu$
for some weight $\mu$ of $\tau$.
The number of fundamental representations and of their weights are finite, 
so Part \eqref{item_lem_lambdapi_eq_lambdarho}
 follows from Part \eqref{item_lem_lambdapi_eq_omega}.
Part \eqref{item_lem_lambdapi_eq_lambdarhos}
is proved recursively from Part \eqref{item_lem_lambdapi_eq_lambdarho}.
\end{proof}

We can now start 
the proof of the implication \eqref{item_lem_diff2_leibRep}$\Rightarrow$\eqref{item_lem_diff2_sigma=} in Lemma \ref{lem_diff2}.
Let $\sigma\in \Sigma$ satisfying $\sigma(1_{\Gh})=0$
and Property \eqref{item_lem_diff2_leibRep} of Lemma \ref{lem_diff2}.

Let us show that $\Op(\sigma):L^2_{fin}(G)\rightarrow \cD(G)$ extends uniquely to a continuous linear operator $H^1(G)\rightarrow L^2(G)$.
Let us recall \cite[Section 2.3]{monJFA}
that the Sobolev spaces $H^s(G)$, $s\geq0$, on $G$ may be defined via local maps and the Euclidean Sobolev spaces on $\bR^n$, 
or globally as the closure of $\cD(G)$ for the norm: 
$$
\|\phi\|_{H^s}:= \|(\id+\cL)^{s/2} \phi\|_{L^2(G)}=\|\phi\|_{H^s}.
$$
Furthermore the space $L^2_{fin}(G)$ is dense in each Hilbert space $H^s$ and in the topological vector space $\cD(G) = \cap_{s\geq 0}H^s = \cap_{s\in \bN}H^s $.
Hence it suffices to show that 
\begin{equation}
\label{eq_bdsigma_(1+lambda)}
\exists C>0
\quad\forall \pi\in\Gh
\quad
\|\sigma(\pi)\|_{\sL(\cH_\pi)}
\leq C (1+\lambda_\pi)^{1/2}.
\end{equation}

Let $\pi\in \Gh$.
Denoting by $\omega_j$ the highest weight of each $\varphi_j\in \FundG$,
we can write the highest weight of $\pi$ as
 $\omega_\pi = m_1 \omega_1+\ldots+m_f\omega_f$
and the representation  $\pi$ occurs in the decomposition of $\varphi^{\otimes m}$
where $m=(m_1,\ldots,m_f)\in \bN_0^f$.
Therefore, we have:
$$
\|\sigma(\pi)\|_{\sL(\cH_\pi)}
\leq
\|\sigma(\varphi^{\otimes m})\|_{\sL(\cH_{\varphi^{\otimes m}})}.
$$
Applying \eqref{eq_lem_diff2_leibRep} recursively and taking the operator norm yield:
\begin{align*}
\|\sigma(\varphi^{\otimes m})\|_{\sL(\cH_{\varphi^{\otimes m}})}
&\leq
m_1 \|\sigma(\varphi_1)\|_{\sL(\cH_{\varphi_1})}
+\ldots+
m_f \|\sigma(\varphi_f)\|_{\sL(\cH_{\varphi_f})}
\\&\leq 
\sum_j m_j \max_{\varphi\in \Fund G} \|\sigma(\varphi)\|_{\sL(\cH_{\varphi})}.
\end{align*}
By Lemma \ref{lem_lambdapi} Parts \eqref{item_lem_lambdapi_eq_omega}
and 
\eqref{item_lem_lambdapi_eq_s},
$\sum_j m_j \asymp (1+\lambda_\pi)^{1/2}$.
We have therefore obtained \eqref{eq_bdsigma_(1+lambda)},
and $\Op(\sigma)$ extends uniquely in a continuous linear operator $H^1(G)\rightarrow L^2(G)$.

Let us show that 
\begin{equation}
\label{eq_opsigmaf1f2}
\forall f_1,f_2\in \cD(G)\qquad
\Op(\sigma)(f_1f_2) (x)=
\Op(\sigma)(f_1)(x) \ f_2(x)
+
f_1(x) \ \Op(\sigma)(f_2)(x).
\end{equation}
We have just proved that the linear operator $\Op(\sigma):H^1(G)\rightarrow L^2(G)$ is continuous.
As $L^2_{fin}(G)$ is a dense subspace of $\cD(G)$ for the $H^1(G)$-norm, it suffices to prove \eqref{eq_opsigmaf1f2} for $f_1,f_2\in L^2_{fin}(G)$, 
and in fact for  $f_1\in L^2_{\pi_1}$ and $f_2\in L^2_{\pi_2}$
for some $\pi_1,\pi_2\in \Gh$, 
and furthermore for  $f_1(x)=\langle\pi_1(x) v_1,w_1\rangle$,
$f_2(x)=\langle\pi_2(x) v_2,w_2\rangle$
for some unit vector $v_1,w_1\in \cH_{\pi_1}$, $v_2,w_2\in \cH_{\pi_2}$.
Let us assume that $f_1,f_2$ are such functions.

The Peter-Weyl theorem (Theorem \ref{thm_PW}) 
implies  readily
\begin{equation}
\label{eq_lem_diff2_opf1}
\widehat f_1 (\pi) :v\longmapsto \delta_{\pi=\pi_1} \frac 1{d_{\pi_1}} 
\langle v, v_1\rangle w_1 \, ,
\quad\mbox{and}\quad
\Op(\sigma)(f_1)(x)
=
\langle  \pi_1(x)\sigma(\pi_1) w_1 , v_1\rangle  \, .
\end{equation}
We decompose $\pi_1\otimes \pi_2=\sum_\rho  \rho $ into a finite sum of $\rho\in \Gh$.
We also decompose the vectors $v_1\otimes v_2=\sum_\rho a_\rho v_\rho$
and $w_1\otimes w_2=\sum_\rho  b_\rho w_\rho$
with $v_\rho,w_\rho$ unit vectors of $\cH_\rho$ and $a_\rho,b_\rho\in \bC$.
We have $f_1f_2=\sum_\rho g_\rho$
with $g_\rho(x)=a_\rho b_\rho \langle \rho(x) v_\rho,w_\rho\rangle$, so 
$$
\Op(\sigma)(f_1f_2)(x) =\sum_\rho \Op(\sigma)(g_\rho)(x)
=\sum_\rho  a_\rho b_\rho
\langle \rho(x)\sigma(\rho) w_\rho , v_\rho\rangle \, ,
$$
having used the computations in \eqref{eq_lem_diff2_opf1}
with $g_\rho$ instead of $f_1$.
Using tensor notation, this can be summarised as
$$
\Op(\sigma)(f_1f_2)(x) =
\langle (\pi_1\otimes \pi_2)(x)\sigma(\pi_1\otimes\pi_2) w_1\otimes w_2
 , v_1\otimes v_2\rangle \, .
$$

We now use $(\pi_1\otimes \pi_2)(x)=\pi_1(x)\otimes \pi_2(x)$
and \eqref{eq_lem_diff2_leibRep} to obtain
\begin{align*}
\Op(\sigma)(f_1f_2)(x) 
&=
\langle \pi_1(x)\sigma(\pi_1)w_1,v_1\rangle
\langle \pi_2(x)\sigma(\pi_2)w_2,v_2\rangle
\\
&\qquad+
\langle \pi_1(x)\sigma(\pi_1)w_1,v_1\rangle
\langle \pi_2(x)\sigma(\pi_2)w_2,v_2\rangle.
\end{align*}
Thanks to \eqref{eq_lem_diff2_opf1},
we recognise the right-hand side of \eqref{eq_opsigmaf1f2} for our choice of functions $f_1,f_2$ above.

This shows that \eqref{eq_opsigmaf1f2} holds for any $f_1,f_2\in \cD(G)$.
In other words, $\Op(\sigma)$ is a derivation of $\cD(G)$.
Therefore, it is a vector field on $G$. 
As the operator $\Op(\sigma)$ is left-invariant, 
it must coincide with some left-invariant vector field  $X$ identified with an element of $\fg_\bC$.
This shows Property \eqref{item_lem_diff2_sigma=}.
 This concludes the proof of Lemma \ref{lem_diff2}.  

\section{Homogeneous Sobolev spaces on $\Gh$}
\label{sec_dotHsGh}

In this section, we define homogeneous Sobolev-type spaces on $\Gh$.
The motivations behind their definitions is not only their formal appeal but also the fact that they have already appeared indirectly for instance in various works of Coifman and Weiss on $SU(2)$ in the 70's, see Remark \ref{remCW}. 
Beside the definitions, basic properties and characterisations, 
we  discuss weak Leibniz estimates, see
 Section \ref{subsec_weak_Leib}.
In Section \ref{sec_f(L)indotHs}, we will discuss the  example of functions of $\widehat \cL$.

\subsection{First definition}
\label{subsec_def_dotHs}

Before defining the homogeneous Sobolev spaces, 
we recall \cite{dixmierVnA}
that on the space of measurable fields of bounded operators over a standard set one can define naturally two important subspaces: the Banach space (in fact, von Neumann algebra) of fields with bounded essential supremum, and the Banach  space (in fact Hilbert space) of square-integrable fields of Hilbert-Schmidt operators. Concretely in the case of $\Sigma_\tau$, this leads to the Banach space $L^\infty(\Gh:\Sigma_\tau)$ given by the norm 
$$
\|\sigma\|_{L^\infty(\Gh:\Sigma_\tau)}
:=
\sup_{\pi \in \Gh} \|\sigma(\pi)\|_{\sL(\cH_\tau\otimes \cH_\pi)}, \qquad \sigma\in \Sigma_\tau, 
$$
and to the Hilbert space $L^2(\Gh:\Sigma_\tau)$ given by the norm
$$
\|\sigma\|_{L^2(\Gh:\Sigma_\tau)}^2
:=
\sum_{\pi\in \Gh} d_\pi \|\sigma(\pi)\|_{HS(\cH_\tau\otimes \cH_\pi)}^2\qquad \sigma\in \Sigma_\tau.
$$
Naturally, in the case $\tau=1_{\Gh}$, we recover the Banach spaces 
$L^\infty(\Gh:\Sigma)$
and $L^2(\Gh:\Sigma)$
defined in Section \ref{subsec_Fmult}.

\begin{definition}
\label{def_dotHsinN}
Let $s\in \bN$.
The homogeneous Sobolev space $\dot H^s(\Gh)$ 
on $\Gh$ 
is the set of  $\sigma\in \Sigma$ such that
$\Delta^\alpha \sigma \in L^2(\Gh, \Sigma_{\varphi^{\otimes \alpha}})$ for every $\alpha\in \bN_0^f$, $|\alpha|=s$.
In this case, the quantity
\begin{equation}
\label{eq_def_homsobnorm}
\|\sigma\|_{\dot H^s(\Gh)}^2 := 
\sum_{|\alpha|=s} \|\Delta^\alpha \sigma \|_{L^2(\Gh, \Sigma_{\varphi^{\otimes \alpha}})}^2,
\end{equation}
is called the homogeneous Sobolev norm.
\end{definition}

\begin{remark}
Note that we will define an equivalent homogeneous Sobolev norm in the case $s=2,3,\ldots$ for which we may use the same notation as in Definition \ref{def_dotHsinN}.
See Definition \ref{def_dotHs} and Remark \ref{rem_def_dotHs}.
\end{remark}

For instance, 
the homogeneous Sobolev space $\dot H^1(\Gh)$ on $\Gh$ 
is the set of symbol $\sigma\in \Sigma$
such that
the following quantity is finite:
$$
\|\sigma\|_{\dot H^1(\Gh)} ^2
:= 
 \sum_{\varphi\in \FundG} 
 \sum_{\pi\in \Gh} d_\pi 
 \| \Delta_\varphi \sigma (\pi)\|_{HS(\cH_\pi \otimes \cH_\varphi)}^2.
$$
The map $\sigma\mapsto \|\sigma\|_{\dot H^1(\Gh)}$ is a semi-norm 
or in other words a non-definite norm
since by Proposition \ref{prop_annilation},
we have for any $\sigma\in \Sigma$, 
$$
\|\sigma\|_{\dot H^1(\Gh)}=0\quad
\Longleftrightarrow \quad \sigma=\sigma(1_{\Gh}) \id.
$$

We can define the associated non-definite Hilbert product
$$
(\sigma_1,\sigma_2)_{\dot H^1(\Gh)}
=
 \sum_{\varphi\in \FundG} 
 \sum_{\pi\in \Gh} d_\pi 
 \tr_{\cH_\pi \otimes \cH_\varphi}
 \left(
 \Delta_\varphi \sigma_1 (\pi)
 \
 (\Delta_\varphi \sigma_2 (\pi))^*
 \right).
 $$
 
More generally, for $s\in \bN$, 
the homogeneous Sobolev space $\dot H^s(\Gh)$ on $\Gh$ 
is the space of $\sigma\in \Sigma$ such that
the semi-norm 
$$
 \|\sigma\|_{\dot H^s(\Gh)}
 =
 \sqrt{
 \sum_{|\alpha|=s} 
\sum_{\pi\in \Gh} d_\pi \|\Delta^\alpha \sigma(\pi) \|_{HS(\cH_{\varphi^{\otimes \alpha}}\otimes \cH_\pi)}^2}
$$
is finite.
By Proposition \ref{prop_annilation}
we have for any $\sigma\in \dot H^s(\Gh)$:
$$
\|\sigma\|_{\dot H^s(\Gh)} =0
\ \Longleftrightarrow \
\sigma\in \cF_G(\fU_{s-1}(\fg_\bC)).
$$
We can define the non-definite Hilbert inner product of $\sigma_1,\sigma_2\in \dot H^s(\Gh)$ via:
$$
 (\sigma_1,\sigma_2)_{\dot H^s(\Gh)}
 =
  \sum_{|\alpha|=s} 
\sum_{\pi\in \Gh} 
d_\pi 
\tr_{HS(\cH_{\varphi^{\otimes \alpha}}\otimes \cH_\pi)}
\left(\Delta^\alpha \sigma_1(\pi)
\ 
(\Delta^\alpha \sigma_2(\pi))^*
\right).
$$

\subsection{First properties of $\dot H^s(\Gh)$, $s\in \bN$}
\label{subsec_prop_sint}

This section is devoted to the proof of the following theorem, as well as its consequences:

\begin{theorem}
\label{thm_dotHs_q1}
\begin{itemize}
\item Case $s=1$:
The homogeneous Sobolev space $\dot H^1(\Gh)$ quotiented by the kernel $\bC \id$ 
of its seminorm
is a Hilbert space isometrically isomorphic to $L^2(q_1)$, 
where $q_1$ is the function 
defined via:
$$
q_1(x):=\sum_{\varphi\in \FundG}\|\varphi(x)-\id\|_{HS(\cH_\varphi)}^2 
=\sum_{\varphi\in \FundG} \sum_{1\leq i,j\leq d_\varphi}
|[\varphi -\id]_{i,j}(x)|^2.
$$
\item Case $s\in \bN$:
More generally, for any $s\in \bN$, 
the homogeneous Sobolev space $\dot H^s(\Gh)$ quotiented by the kernel $\cF_G(\fU_{s-1}(\fg_\bC))$
of its seminorm
is a Hilbert space isometrically isomorphic to $L^2(q_s)$ where 
$$
q_s(x)=\sum_{|\alpha|=s}\sum_{i,j} 
\Pi_{\ell=1}^f |[\varphi_\ell(x)-\id]_{i_{1,\ell},j_{1,\ell}}\ldots 
[\varphi_\ell(x)-\id]_{i_{\alpha_\ell,\ell},j_{\alpha_\ell,\ell}} |^2,
$$
the sum $\sum_{i,j}$ being a shorthand for  $i_{1,\ell},j_{1,\ell},\ldots, i_{\alpha_\ell,\ell},j_{\alpha_\ell,\ell}$.
\end{itemize}
\end{theorem}
	
As is customary, for $p\in [1,+\infty)$ and  $\omega:G\rightarrow\bC$ a non-negative measurable function,
$L^p(\omega(x)dx)$ denotes the Banach space of measurable function $f:G\rightarrow \bC$ such that $\omega(x) |f(x)|^p$ is integrable.

\medskip

The proof of Theorem \ref{thm_dotHs_q1} will follow from the statements of this section.
Let us first study the function $q_1$ and the case $s=1$:

\begin{proposition}
\label{prop_q1}
\begin{enumerate}
\item 
	The function $q_1$ is smooth, non-negative, vanishes only at $e_G$ and is equivalent to $|x|^2$ in the sense that
$$
\exists C>0\quad\forall x\in G\qquad
C|x|^2\leq q_1(x)\leq C^{-1} |x|^2.
$$
Consequently, $L^2(q_1)=L^2(|x|^{2} dx)$.
\item 
The group Fourier transform of a smooth function $f\in \cD(G)$ is in $\dot H^1(\Gh)$ 
and we have:
$$
\|\widehat f\|_{\dot H^1(\Gh)} = \|f\|_{L^2(q_1)}
$$
\item 
The group Fourier transform of a distribution   $f\in \cD'(G)$ is in $\dot H^1(\Gh)$ 
if and only if $f\in L^2(|x|^2)$. 
Conversely, for every $\sigma\in \dot H^1(\Gh)$
there exists a unique function $f$ locally integrable on $G\backslash\{e_G\}$ satisfying $f\in L^2(q_1)$ 
with $\|f\|_{L^2(q_1)}=\|\sigma\|_{\dot H^1(\Gh)}$ 
and
$$
\forall \phi\in \cD(G)
\qquad
(\phi,f)_{L^2(q_1)} = 
(\widehat \phi,\sigma)_{\dot H^1(\Gh)}.
$$
\end{enumerate}
\end{proposition}

The ideas of the proof of Proposition \ref{prop_q1} Part (1)
are essentially the same as in 
\cite[Lemma 5.11]{monJFA}.
\begin{proof}[Proof of Proposition \ref{prop_q1} Part (1)]
Clearly $q_1$ is a non-negative smooth function.
As the $q^{(\tau)}_{i,j}$'s vanish at $e_G$ so does $q_1$.
Moreover we have
$$
q_1(x_0)=0
\quad \Longleftrightarrow  \quad
\forall \varphi\in \FundG  \quad \varphi(x_0) =\id_\varphi
\quad \Longleftrightarrow  \quad
\forall \tau\in \RepG  \quad \tau(x_0) =\id_\tau, 
$$
since any representation on $G$ occurs in the decomposition of tensor of fundamental representation.
This together with the inversion formula shows that a zero $x_0$ of $q_1$ must satisfy $f(x_0)=f(e_G)$
for any $f\in \cD(G)$,  
and therefore  $x_0=e_G$.
In other words, $e_G$ is the only zero of $q_1$.

Having fixed an orthonormal basis $\{X_1,\ldots, X_n\}$ of $\fg$, 
we compute easily
$$
X_k q_1 =\sum_{\tau\in \FundG}
\sum_{1\leq i,j\leq d_\tau} 
q^{(\tau)}_{i,j} \, X_k \bar q^{(\tau)}_{i,j}
+
\bar q^{(\tau)}_{i,j}\, X_k q^{(\tau)}_{i,j},
$$
so that $X_k q_1 (e_G)=0$, 
and 
$$
X_\ell X_k q_1 (e_G) =
2\Re \sum_{\tau\in \FundG}
\sum_{1\leq i,j\leq d_\tau} 
 \tau_{i,j}  (X_\ell) \,  \bar \tau_{i,j}(X_k).
$$
Let us show that the Hessian matrix ${\rm Hess} (q_1)(e_G) = [X_\ell X_k q_1 (e_G)]_{1\leq \ell ,k\leq n}$ is positive definite. 
If $v=(v_1,\ldots,v_n)^t\in \bR^n$, 
then 
\begin{eqnarray*}
({\rm Hess} (q_1)(e_G) v) \cdot v
&=&
\sum_{1\leq \ell ,k\leq n}X_\ell X_k q_1 (e_G) v_\ell v_k
\\&=&
2\Re \sum_{\tau\in \FundG}
\sum_{1\leq i,j\leq d_\tau} 
\sum_{1\leq \ell ,k\leq n}
 \tau_{i,j}  (X_\ell) \,  \bar \tau_{i,j}(X_k)
\, v_\ell v_k
\\&=&
2\Re \sum_{\tau\in \FundG}
\sum_{1\leq i,j\leq d_\tau} 
|\sum_{1\leq \ell \leq n} v_\ell \tau_{i,j}  (X_\ell )|^2
\\&=&
 \sum_{\tau\in \FundG}
\|\tau(\sum_{1\leq \ell \leq n} v_\ell X_\ell)\|^2_{HS(\cH_\tau)}.
\end{eqnarray*}
So this quantity is non-negative. 
Furthermore if it is zero for some $v\in \bR^n$, then 
the vector
$\sum_{1\leq \ell \leq n} v_\ell X_\ell \in \fg$ 
is in the kernel of every infinitesimal fundamental representation, 
therefore of any representation of the reductive Lie algebra $\fg$;
hence this vector has to be trivial and $v=0$.
This shows that ${\rm Hess} (q_1)(e_G)$ is positive definite and 
concludes the proof.
\end{proof}

\begin{proof}[Proof of Proposition \ref{prop_q1} Part (2)]
For any $f\in \cD(G)$, $\varphi\in \FundG$ and $1\leq i,j\leq d_\varphi$, 
we observe that $\cF_G([\varphi-\id]_{i,j} f)$ is the $i,j$-component of $ \Delta_\varphi \widehat f (\pi)$ in the $\cH_\varphi$-part of the tensor.
So the Plancherel formula implies
$$
\sum_{1\leq i,j\leq d_\varphi} 
\|[\varphi-\id]_{i,j} f\|_{L^2(G)}^2
=\sum_{\pi\in \Gh} d_\pi 
\| \Delta_\varphi \widehat f (\pi)\|_{HS(\cH_\varphi\otimes\cH_\pi)}^2
.
$$
Summing over $\varphi\in \FundG$, one obtains
$\|f\|_{L^2(q_1)}^2$ on the left hand-side
and $\|\widehat f\|_{\dot H^1(\Gh)}$ on the right.
\end{proof}

\begin{proof}[Proof of Proposition \ref{prop_q1} Part (3)]
Part (2) implies that the map $\phi\mapsto(\widehat \phi,\sigma)_{\dot H^1(\Gh)}$ 
is linear on $\cD(G)$ and satisfies
$$
|(\widehat \phi,\sigma)_{\dot H^1(\Gh)}|\leq
\|\widehat \phi\|_{\dot H^1(\Gh)}
\|\sigma\|_{\dot H^1(\Gh)}
=
 \|\phi\|_{L^2(q_1)}
 \|\sigma\|_{\dot H^1(\Gh)}.
$$
Therefore this linear form extends continuously on the Hilbert space 
$L^2(q)$
and is given by $\phi\mapsto (\phi,f)_{L^2(q_1)}$
for a unique $f\in L^2(q_1)$.
Moreover $\|f\|_{L^2(q_1)}$ is equal to the norm of this linear form, 
which is $\leq \|\sigma\|_{\dot H^1(\Gh)}$.
One shows that $\|f\|_{L^2(q_1)}= \|\sigma\|_{\dot H^1(\Gh)}$
by  considering the sequence of functions 
$\phi_\ell$ 
such that $\widehat \phi_\ell (\pi)= \sigma (\pi)$
for $\lambda_\pi\leq \ell$ and 0 otherwise
and let $\ell\in \bN$ tends to $+\infty$.
The rest of the proof is given by routine arguments.
\end{proof}

\begin{remark}
\label{remCW}
Note that in the case of $SU(2)$, 
there is only one fundamental representation $\varphi$;
it is the linear action on $\bC^2$, or in other word $\varphi(g)=g$ for all $g\in SU(2)$.
It is then easy to see that 
the $L^2(\Gh:\Sigma)$ norm of the difference operator $\Delta^2$  defined in \cite[Ch. IV \S3 and Ch. V]{coifman+weiss} applied to a symbol equals the norm $\dot H^1(\Gh)$ of the symbol.
\end{remark}

As in the case $s=1$, we prove easily:
\begin{lemma}
\label{lem_dotHs_L2}
Let $s\in \bN$.
\begin{enumerate}
\item 
The function $q_s$ is smooth, non-negative, vanishes only at $e_G$ and is equivalent to $q_1^s$ and to $|x|^{2s}$.
Consequently, $L^2(q_s(x) dx)=L^2(q_1^s(x) dx)=L^2(|x|^{2s} dx)$ for any $s\in \bR$.

\item	
For every $f\in \cD(G)$, we have:
$$
\|f\|_{L^2(q_s)}^2=\|\widehat f\|_{\dot H^s(\Gh)}^2.
$$
\item 
For  every $\sigma\in \dot H^s(\Gh)$
there exists a unique function $f$ locally integrable on $G\backslash\{e_G\}$ satisfying $f\in L^2(q_1^s(x)dx)$ 
with $\|f\|_{L^2(q_1^s(x)dx)}=\|\sigma\|_{\dot H^s(\Gh)}$ 
and
$$
\forall \phi\in \cD(G)
\qquad
(\phi,f)_{L^2(q_1^s(x)dx)} = 
(\widehat \phi,\sigma)_{\dot H^s(\Gh)}.
$$
\end{enumerate}
\end{lemma}

\begin{proof}
Part (1) is easily proved. For Part (2), 
the same arguments as in the proof of Proposition \ref{prop_q1} Part (2) show:
$$
\sum_{|\alpha|=s}\sum_{i,j} 
\Pi_{\ell=1}^f |[\varphi_\ell(x)-\id]_{i_{1,\ell},j_{1,\ell}}\ldots 
[\varphi_\ell(x)-\id]_{i_{\alpha_\ell,\ell},j_{\alpha_\ell,\ell}} |^2
=\sum_{|\alpha|=s}\sum_{\pi\in \Gh} d_\pi 
\| \Delta^\alpha \widehat f (\pi)\|_{HS(\cH_{\varphi^{\otimes\alpha}}\otimes\cH_\pi)}^2
.
$$
We recognise
$\|f\|_{L^2(q_1)}^2$ on the left hand-side
and $\|\widehat f\|_{\dot H^s(\Gh)}$ on the right.
We conclude in the same way as in the proof of Proposition \ref{prop_q1}.
\end{proof}

This finishes the proof of Theorem \ref{thm_dotHs_q1}.

In Parts (3) of Proposition \ref{prop_q1} and Lemma \ref{lem_dotHs_L2}, we may abuse the notation 
and set
$$
f:=\cF_G^{-1}\sigma.
$$
This allows us to say that the isometry used in the statement is given by the Fourier transform and its extension.

\subsection{Definition and properties of $\dot H^s(\Gh)$, $s\in \bR$}
From the results 
in Section \ref{subsec_prop_sint}, we see that the space $\dot H^s(\Gh)$ are growing with $s\in \bN$, 
and that there is a natural extension of the definition for $s\in \bR$:

\begin{definition}
\label{def_dotHs}
Let $s\in \bR$.
The homogeneous Sobolev space $\dot H^s(\Gh)$
is defined as 
\begin{itemize}
\item the set of symbol $\sigma\in L^2(\Gh:\Sigma)$
such that $\cF_G^{-1}\sigma\in L^2(|x|^s dx)$  if $s\leq 0$,
\item in Definition \ref{def_dotHsinN} if $s> 0$ and $s\in \bN$, 
\item the set of symbol $\sigma\in \dot H^{[s]+1}(\Gh)$ such that 
$\cF_G^{-1}\sigma\in L^2(|x|^s dx)$
if $s>0$ and $s\not \in \bN$.
\end{itemize}

In all cases, we can define the non-definite norm as
$$
\|\sigma\|_{\dot H^s(\Gh)}:= \|\cF_G^{-1}\sigma\|_{L^2(q_1^s(x)dx)}.
$$
\end{definition}

Here $[s]$ denotes the integer part of $s>0$, that is, the largest integer $s_0\geq 0$ such that $s_0\leq s$.

\begin{remark}
\label{rem_def_dotHs}	
In the cases $s=2,3,\ldots$,    in Definitions \ref{def_dotHsinN} and \ref{def_dotHs}, we use 
the same notation
$\|\cdot\|_{\dot H^s(\Gh)}$ for two generally different semi-norms;
we allow ourselves this abuse of notation  since the two semi-norms are equivalent by Lemma \ref{lem_dotHs_L2}. 
\end{remark}

These homogeneous Sobolev spaces $\dot H^s(\Gh)$ enjoy the following properties:
\begin{proposition}
\label{prop_dotHs}
Let $s\in \bR$.

\begin{enumerate}
\item 
The quotient of 
$\dot H^s(\Gh)$ by the kernel of its norm 
is a Hilbert space isometrically isomorphic to 
$L^2(q_1^s)$,
and we have
$L^2(q_1^s)=L^2(|x|^{2s} dx)$.
The kernel of its norm is 0 if $s\leq 0$, 
$\cF_G(\fU_{s-1}(\fg_\bC))$ if $s\in \bN$, 
and 
$\cF_G(\fU_{[s]-1}(\fg_\bC))$ if $s\in (0,+\infty)\backslash\bN$.
\item
\label{item_prop_dotHs_subset} 
For all $s_2 \geq s_1$, 
we have $\dot H^{s_2}(\Gh)
\subset 
\dot H^{s_1}(\Gh)$
and there exists $C>0$
such that
$$
\forall \sigma\in \dot H^{s_2}(\Gh) 
\qquad
\|\sigma\|_{\dot H^{s_1}(\Gh)}
\leq C \|\sigma\|_{\dot H^{s_2}(\Gh)}.
$$
\item
If $s<n/2$ and $\sigma\in \dot H^s(\Gh)$
then $\cF_G^{-1} \sigma \in L^1(G)$
with 
$\|\cF_G^{-1} \sigma \|_{L^1(G)}\leq C_s \|\sigma\|_{\dot H^s(\Gh)}$.
\item 
\label{item_interpolation_prop_dotHs}
Let $s_1,s_2\in \bR$.
For any $\sigma\in \dot H^{s_1}(\Gh) \cap \dot H^{s_2}(\Gh)$, 
we have
 $\sigma\in \dot H^{s}(\Gh)$
for any $s\in [s_1,s_2]$ and if we write $s$ as 
$s= \theta s_1 +(1-\theta)s_2$
with $\theta\in [0,1]$ then we have the inequality:
$$
\|\sigma\|_{\dot H^{s}(\Gh)}
\leq
\|\sigma\|_{\dot H^{s_1}(\Gh)}^\theta
\|\sigma\|_{\dot H^{s_2}(\Gh)}^{1-\theta}.
$$\end{enumerate}
\end{proposition}

\begin{proof}
The properties follow readily from Section \ref{subsec_def_dotHs}, 
and 
\begin{itemize}
\item 
the inclusion $L^2(|x|^{2s_2}dx)\subset L^2(|x|^{2s_1}dx)$
holds and is continuous when $s_1\geq s_2$, 
\item 
the Cauchy-Schwartz inequality 
on $G$ implies that 
$$
\forall s_1,s_2\in \bR,  \ \theta\in [0,1]\qquad
\|f\|_{L^2(|x|^{2(\theta s_1 +(1-\theta)s_2)}dx)}
\leq
\|f\|_{L^2(|x|^{2s_1}dx)}^\theta
\|f\|_{L^2(|x|^{2s_2}dx)}^{1-\theta},
$$
and that 
$$
\|f\|_{L^1(G)}\leq \||\cdot|^{-s}\|_{L^2} \||\cdot|^{s} f_{L^2}.
$$
\end{itemize}
\end{proof}

\subsection{The spaces $\dot L^\infty_s(\Gh:\Sigma)$}
\label{subsec_dotLpsinN}

In this section
we define and study  subspaces $\dot L^\infty_s(\Gh:\Sigma)$ of $\Sigma$
which will be useful in the study of Leibniz-type estimates in Section
\ref{subsec_weak_Leib}.

\begin{definition}
\label{def_dotLpsinN}
Let $s\in \bN$.
The space $\dot L^\infty_s(\Gh:\Sigma)$
is the set of symbol $\sigma\in \Sigma$ such that
$\Delta^\alpha \sigma \in L^\infty(\Gh, \Sigma_{\varphi^{\otimes \alpha}})$ for every $\alpha\in \bN_0^f$, $|\alpha|=s$.
In this case, we define the quantity
$$
\|\sigma\|_{\dot L^\infty_s(\Gh:\Sigma)} := 
\max_{|\alpha|=s} \|\Delta^\alpha \sigma \|_{L^\infty(\Gh, \Sigma_{\varphi^{\otimes \alpha}})},
$$
\end{definition}

The map $\sigma\mapsto \|\sigma\|_{\dot L^\infty_s(\Gh:\Sigma)}$ is a semi-norm 
or in other words a non-definite norm
since by Proposition \ref{prop_annilation},
we have for any $\sigma\in \Sigma$, 
$$
\|\sigma\|_{L^\infty_s(\Gh)}=0
\quad\Longleftrightarrow \quad
\sigma\in \cF_G(\fU_{s-1}(\fg_\bC)).
$$

We define the case $s=0$ as coinciding with $L^\infty(\Gh:\Sigma)$:
$$
\dot L^\infty_0(\Gh:\Sigma):= L^\infty(\Gh:\Sigma) 
\quad\mbox{and}\quad
\|\sigma\|_{\dot L^\infty_0(\Gh)}:=\|\sigma\|_{L^\infty(\Gh:\Sigma)}.
$$

The following shows that the Fourier transform of integrable functions are in 
 $L^\infty_s(\Gh)$:
\begin{lemma}
\label{lem_dotLinftys_L1G}
For any $s\in \bN_0$, there exists $C>0$ such that
$$
\forall\kappa\in L^1(G)
\qquad
\|\widehat \kappa\|_{\dot L^\infty_s(\Gh:\Sigma)}
\leq C \| \kappa\|_{L^1(|x|^s dx)}.
$$
\end{lemma}
\begin{proof}
Let $\kappa\in L^1(G)$. 
We have
$$
\Delta_\varphi \widehat \kappa (\pi)
=
\widehat \kappa (\varphi\otimes\pi)
-\widehat \kappa (\id_\varphi\otimes\pi)
=
\int_G\kappa(x) \left(\varphi(x)-\id\right)^*\otimes \pi(x) dx,
$$
so
\begin{align*}
\|\Delta_\varphi \widehat \kappa (\pi)\|_{\sL(\cH_\pi\otimes \cH_\varphi)}
&\leq
\int_G|\kappa(x)|
\|
 \left(\varphi(x)-\id\right)^*\otimes \pi(x)\|_{\sL(\cH_\pi\otimes \cH_\varphi)}
 dx
 \\&\leq
\int_G|\kappa(x)|
\|\varphi(x)-\id\|_{\sL(\cH_\varphi)} dx.
\end{align*}
Now all the operator norms 
on the finite dimensional space $\cH_\varphi$
are equivalent, so we may replace the $\sL$-operator norm 
with the Hilbert Schmidt norm:
$$
\|\varphi(x)-\id\|_{\sL(\cH_\varphi)}
\asymp
\|\varphi(x)-\id\|_{HS(\cH_\varphi)},
$$
so that 
$$
\|\Delta_\varphi \widehat \kappa (\pi)\|_{\sL(\cH_\pi\otimes \cH_\varphi)}
\lesssim 
\int_G|\kappa(x)|
\|\varphi(x)-\id\|_{HS(\cH_\varphi)} dx.
$$
Taking the maximum over the finite number of fundamental representations, we obtain:
$$
\|\widehat \kappa (\pi)\|_{\dot L^\infty_1(\Gh,\Sigma)}
\lesssim 
\int_G|\kappa(x)| \tilde q(x)dx,
$$
with $\tilde q(x):=\max_{\varphi\in \FundG} \|\varphi(x)-\id\|_{HS(\cH_\varphi)}$. 
As there are only a finite number of fundamental representations, $\tilde q(x)\asymp \sqrt{q_1(x)} \asymp |x|$, see Section \ref{subsec_prop_sint}.
This shows the case $s=1$.

The general case is proved in a similar way.
\end{proof}

\begin{lemma}
For any $s_1,s_2\in \bN_0$ with $s_1\leq s_2$
the inclusion $\dot L^\infty_{s_1}(\Gh : \Sigma) \subset \dot L^\infty_{s_2}(\Gh)$
holds and there exists $C>0$ such that
$$
\forall \sigma \in \dot L^\infty_{s_1}(\Gh : \Sigma)
\quad
\|\sigma\|_{\dot L^\infty_{s_2}(\Gh)}
\leq
C
\|\sigma\|_{\dot L^\infty_{s_1}(\Gh : \Sigma)}.
$$
\end{lemma}

\begin{proof}
We observe that 
$$
\|\Delta_\varphi \sigma(\pi)\|_{\sL(\cH_\varphi\otimes \cH_\pi)}
\leq
\|\sigma(\varphi\otimes \pi)\|_{\sL(\cH_\varphi\otimes \cH_\pi)}
+
\|\id_\varphi\otimes \sigma(\pi)\|_{\sL(\cH_\varphi\otimes \cH_\pi)}
\leq
2 \|\sigma\|_{L^\infty(\Gh:\Sigma)}.
$$
This implies the case $s_1=0$, $s_2=1$.

The general case is proved using similar considerations.
\end{proof}

Thanks to the introduction of $\dot L^\infty_s(\Gh:\Sigma)$,
we can now  state and prove Leibniz-type estimates.

\subsection{Weak Leibniz estimates}
\label{subsec_weak_Leib}

Here we collect some results 
about  estimates in homogeneous Sobolev norms of the products of two symbols.
We will use the space $\dot L^\infty_s(\Gh:\Sigma)$, $s\in \bN_0$
defined in Section \ref{subsec_dotLpsinN}.
It will be convenient to modify slightly our previous notation for the homogeneous Sobolev spaces
$$
\dot L^2_s(\Gh : \Sigma) := \dot H^s(\Gh)
\ \mbox{for any}\ s\in \bN_0.
$$

\begin{proposition}
\label{prop_Leibniz_estimate}
For any $s\in \bN_0$
and $p=2,\infty$, 
there exists $C=C_{p,s}$ such that for any $\sigma_1,\sigma_2\in \Sigma$
$$
\|\sigma_1\sigma_2\|_{\dot L^p_s(\Gh : \Sigma)}
\leq
C\sum_{s_1+s_2=s}
\|\sigma_1\|_{\dot L^\infty_{s_1}(\Gh : \Sigma)}
\|\sigma_2\|_{\dot L^p_{s_2}(\Gh : \Sigma)},
$$
in the sense that if the right-hand side is finite then the left-hand side is finite and the inequality holds.
\end{proposition}

\begin{proof}
Using the Leibniz-type formula \eqref{eq_Leibniz}, 
we obtain easily with 
$\|\cdot\|=\|\cdot\|_{HS(\cH_\varphi\otimes \cH_\pi)}$
or 
$\|\cdot\|=\|\cdot\|_{\sL(\cH_\varphi\otimes \cH_\pi)}$:
\begin{align*}
\|\Delta_\varphi (\sigma_1\sigma_2)(\pi)\|
&\leq
\| \Delta_\varphi (\sigma_1) (\pi)\|_{\sL(\cH_\varphi\otimes \cH_\pi)}
\|\sigma_2(\id_\varphi\otimes\pi)\|
+
\|\sigma_1 (\varphi\otimes\pi)\|_{\sL(\cH_\varphi\otimes \cH_\pi)}
\|\Delta_\varphi \sigma_2(\pi)\|
\\
&\leq
\|\sigma_1\|_{\dot L^\infty_1(\Gh : \Sigma)}
\|\sigma_2(\id_\varphi\otimes\pi)\|
+
\|\sigma_1\|_{L^\infty(\Gh:\Sigma)}
\|\Delta_\varphi \sigma_2(\pi)\|.
\end{align*}
This implies for $p=2$ or $\infty$
$$
\|\sigma_1\sigma_2\|_{\dot L^p_1(\Gh : \Sigma)}
\lesssim 
\|\sigma_1\|_{\dot L^\infty_1(\Gh : \Sigma)}
\|\sigma_2\|_{L^p(\Gh : \Sigma)}
+
\|\sigma_1\|_{L^\infty(\Gh:\Sigma)}
\|\sigma_2\|_{\dot L^p_1(\Gh : \Sigma)}.
$$
This shows the case $s=1$. 
The cases $s>1$ are proved in a similar way.
\end{proof}

We remark that a different type of Leibniz estimate can be obtained using the isometry of $\dot H^s$,  modulo the kernel of its norm,
with $L^2(q^s_2(x)dx) =L^2(|x|^{2s}dx) $ and the following lemma.
However, this requires to know that the symbols are the Fourier transform of e.g. integrable functions: 

\begin{lemma}
\label{lem_veryweakLeib}
Let $s>0$. Then there exists $C_s>0$ such that 
for any  two locally integrable functions $f_1,f_2:G\rightarrow \bC$
we have
$$
\|f_1*f_2\| _{L^2(|x|^{2s}dx)}
\leq C_s\left(
\|f_1\| _{L^2(|x|^{2s}dx)}
\|f_2\| _{L^1(G)}
+
\|f_1\| _{L^2(G)}
\|f_2\| _{L^1(|x|^{s}dx)}\right),
$$
and
$$
\|f_1*f_2\| _{L^2(|x|^{2s}dx)}
\leq C_s
\left(
\|f_1\|_{L^1(G)}
 \|f_2\| _{L^2(|x|^{2s}dx)}
+
\|f_1\|_{L^1(|x|^{s}dx)}
\|f_2\|_{L^2(G)}\right).
$$

\end{lemma}

\begin{proof}
Using the triangle inequality for $|\cdot|$ and 
the estimate
$$
\forall s> 0
\qquad
\exists C=C_s>0
\qquad
\forall a,b\geq 0
 \qquad
 (a+b)^s\leq C (a^s +b^s), 
 $$
we obtain
\begin{align*}
\|f_1*f_2\| _{L^2(|x|^{2s}dx)}
&\leq
\| |\cdot|^s (|f_1|*|f_2|)\|_{L^2(G)}
\\&\leq
C\left( \| (|\cdot|^s |f_1|)*|f_2|\|_{L^2(G)}
+
\| |f_1|*(|\cdot|^s |f_2|)\|_{L^2(G)}\right),
\end{align*}
and we conclude with Young's convolution inequalities.
\end{proof}

In the proof of  Lemma \ref{lem_technical_lem_prop_f(L)}, 
we will need to estimate the Sobolev norm of $k$-product of the same symbol
and obtain an estimate with a precise dependence in $k$.
We will use the following technical lemma:

\begin{lemma}
\label{lem_sigmak}
Let $\sigma\in \Sigma$ be such that
for any $\varphi\in \FundG$ and any $\pi\in\RepG$,
$\sigma(\varphi\otimes \pi)$ commutes with 
$\sigma(\id_\varphi\otimes \pi)$.
Let $k\in \bN$.

\begin{enumerate}
\item For any $\varphi\in \FundG$ and any $\pi\in\RepG$,
we have
$$
\Delta_\varphi\sigma^k(\pi)
=
\Delta_\varphi\sigma(\pi)
\sum_{j=0}^{k-1}
\sigma^{j}(\varphi\otimes \pi)
\sigma^{k-1-j}(\id_\varphi\otimes \pi).
$$
\item 
For $p=2$ or $\infty$, we have:
$$
\|\sigma^k\|_{\dot L^p_1(\Gh : \Sigma)}
\leq
\|\sigma\|_{\dot L^p_1(\Gh : \Sigma)}
k \|\sigma\|_{\dot L^\infty(\Gh:\Sigma)}^{k-1}.
$$
\item 
Then for $p=2$ or $\infty$, 
the norm $\|\sigma^k\|_{\dot L^p_s(\Gh : \Sigma)}$ is bounded 
up to a constant of $s=2,3\ldots$
by
\begin{align*}
\|\sigma\|_{\dot L^p_s(\Gh : \Sigma)}
k \|\sigma\|_{\dot L^\infty(\Gh:\Sigma)}^{k-1}
+
\|\sigma\|_{\dot L^p_{s-1}(\Gh : \Sigma)}
\sum_{\substack{s_1+s_2=1\\ j=0,\ldots,k-1}}
\|\sigma^j\|_{\dot L^\infty_{s_1}(\Gh : \Sigma)}
\|\sigma^{k-1-j}\|_{\dot L^\infty_{s_2}(\Gh : \Sigma)}
+\ldots\\
+
\|\sigma\|_{\dot L^p_{s-\ell}(\Gh : \Sigma)}
\sum_{\substack{s_1+s_2=\ell \\ j=0,\ldots,k-1}}
\|\sigma^j\|_{\dot L^\infty_{s_1}(\Gh : \Sigma)}
\|\sigma^{k-1-j}\|_{\dot L^\infty_{s_2}(\Gh : \Sigma)}
+\ldots\\
+
\|\sigma\|_{\dot L^p_1(\Gh : \Sigma)}
\sum_{\substack{s_1+s_2=s-1\\ j=0,\ldots,k-1}}
\|\sigma^j\|_{\dot L^\infty_{s_1}(\Gh : \Sigma)}
\|\sigma^{k-1-j}\|_{\dot L^\infty_{s_2}(\Gh : \Sigma)}.
\end{align*}
\end{enumerate}
\end{lemma}

\begin{proof}
Let $\varphi\in \FundG$ and $\pi\in \RepG$. 
Then 
\begin{align*}
\Delta_\varphi\sigma^k(\pi)
&=
\sigma^k(\varphi\otimes \pi)
-
\sigma^k(\id_\varphi\otimes \pi)
\\&=
(\sigma(\varphi\otimes \pi)
-
\sigma(\id_\varphi\otimes \pi))
\sum_{j=0}^{k-1}
\sigma^{j}(\varphi\otimes \pi)
\sigma^{k-1-j}(\id_\varphi\otimes \pi),
\end{align*}
and this yields Part (1).

In the formula of Part (1), 
applying the norm 
$\|\cdot\|=\|\cdot\|_{HS(\cH_\varphi\otimes \cH_\pi)}$
or 
$\|\cdot\|=\|\cdot\|_{\sL(\cH_\varphi\otimes \cH_\pi)}$, 
we obtain:
$$
\|\Delta_\varphi\sigma^k(\pi)\|
\leq
\|\Delta_\varphi\sigma(\pi)\|\
\|\sum_{j=0}^{k-1}
\sigma^{j}(\varphi\otimes \pi)
\sigma^{k-1-j}(\id_\varphi\otimes \pi)\|_{\sL(\cH_\varphi\otimes \cH_\pi)}
$$
and we easily check that 
$$
\|\sum_{j=0}^{k-1}
\sigma^{j}(\varphi\otimes \pi)
\sigma^{k-1-j}(\id_\varphi\otimes \pi)\|_{\sL(\cH_\varphi\otimes \cH_\pi)}
\leq
\sum_{j=0}^{k-1}
\|\sigma(\varphi\otimes \pi)\|_{\sL(\cH_\varphi\otimes \cH_\pi)}^{j}
\|\sigma(\pi)\|_{\sL(\cH_\pi)}^{k-1-j}
$$
is bounded by $ k \|\sigma\|_{L^\infty(\Gh:\Sigma)}$.
This yields Part (2).

Using Part (1) and the Leibniz type formula  \eqref{eq_Leibniz}
and proceeding as in Part (2) and in the proof of Proposition \ref{prop_Leibniz_estimate}, 
Part (3) follows.
\end{proof}

\begin{remark}
\label{rem_lem_sigmak}
The hypotheses on $\sigma\in \Sigma$ in Lemma \ref{lem_sigmak}, 
namely that 
$\sigma(\varphi\otimes \pi)$ commutes with 
$\sigma(\id_\varphi\otimes \pi)$ for any $\varphi\in \FundG$ and any $\pi\in\RepG$,
are satisfied when $\sigma$ is central but also when $\sigma$ is a function of the group Fourier transform of a left-invariant sub-Laplacian on $G$.
In any case, they are satisfied when $\sigma$ is  
a function of the group Fourier transform of the Laplace-beltrami operator on $G$.
\end{remark}

\section{Function of $\widehat\cL$ in $\dot H^s(\Gh)$}
\label{sec_f(L)indotHs}

In this section, we show that certain types of functions of $\widehat\cL$
are in $\dot H^s(\Gh)$. 
The methods and ideas presented here are  classical, 
see e.g. \cite[Section 1]{mauceri+meda}.

\subsection{Statements}

We will give two types of conditions on the functions:  a first one of H\"ormander type and the second one using Euclidean Sobolev spaces.
More precisely, we will show the following two propositions:
\begin{proposition}
\label{prop_f(L)_sup}
Let $G$ be a connected compact Lie group of dimension $n$.
\begin{enumerate}
\item Let $s>0$. 
There exist $C=C_{s,G}>0$ and $d\in \bN_0$ such that 
for any continuous function  $f:\bR\to \bC$ supported in $[0,1]$ 
we have
$$
\forall t\in (0,1)\qquad
\|f (t\widehat{\cL} )\|_{\dot H^s(\Gh)}
\leq C
t^{\frac 12 (s-\frac n2)} 
\sup_{\lambda\geq 0, \ell =0,\ldots,d}
|f^{(\ell)}(\lambda)|,
$$
in the sense that if the left-hand side is finite then $f (t\widehat{\cL} )\in \dot H^s(\Gh)$ for any $t\in (0,1)$ and the inequality holds.
\item Let $s>n/2$.
There exist $C=C_{s,G}>0$ and $d\in \bN_0$ such that 
for any continuous function  $f:\bR\to \bC$
we have
$$
\|f (\widehat{\cL} )\|_{\dot H^s(\Gh)}
\leq C
\sup_{\lambda\geq 0, \ell =0,\ldots,d}(1+\lambda)^\ell
|f^{(\ell)}(\lambda)|,
$$
in the sense that if the left-hand side is finite then $f (\widehat{\cL} )\in \dot H^s(\Gh)$ and the inequality holds.
\end{enumerate}
\end{proposition}

The proof of Proposition \ref{prop_f(L)_sup} relies on the properties of the $\dot H^s(G)$ and of the kernel of $f(t\cL)$; 
it is given in Section \ref{subsec_pf_prop_f(L)_sup}.

\begin{proposition}
\label{prop_f(L)}
Let $G$ be a compact Lie group of dimension $n$.
Let $s'>s>\frac n2$.
There exist a constant $C>0$ and a function $\eta\in \cD(0,\infty)$
such that for every continuous function $f:\bR\to \bC$, 
we have
$$
\|f (\widehat{\cL} )\|_{\dot H^s(\Gh)}
\leq C  \sup_{r> 0} \|f (r \, \cdot\, ) \eta\|_{H^{s'}(\bR)},
$$
in the sense that if the left-hand side is finite then $f (\widehat{\cL} )\in \dot H^s(\Gh)$ and the inequality holds.
\end{proposition}

The proof of Proposition \ref{prop_f(L)}  will be given in Section \ref{subsec_pf_prop_f(L)}.

\subsection{Proof of Proposition \ref{prop_f(L)_sup}}
\label{subsec_pf_prop_f(L)_sup}

Before proving Proposition \ref{prop_f(L)_sup}, 
we recall some well known estimates for the heat kernel 
and other kernels of spectral multipliers of the Laplace-Beltrami operator $\cL$:

\begin{theorem}
\label{thm_kernelf(L)}
\begin{enumerate}
\item 
For each $t>0$, the heat kernel $p_t:=e^{-t\cL}\delta_{e_G}$
is a positive smooth function on $G$ which satisfies 
$$
\forall s,t>0\qquad
\int_G p_t(x) dx=1,
\quad p_t(x^{-1})=p_t(x),
\quad\mbox{and}\quad p_t*p_s=p_{t+s}.
$$
and there exists $C>0$ such that
$$
\forall x\in G, \ t>0\qquad
0<p_t(x)\leq C V(\sqrt t)^{-1} e^{-\frac{|x|^2}{Ct}}.
$$
\item
For any $s\geq 0$and any $\alpha\in \bN_0^n$, 
there exists a constant $C>0$
such that 
$$
\forall t>0\qquad 
\| |x|^s p_t\|_{L^1(G)}
\leq C \sqrt{t}^{s}
\quad\mbox{and}\quad 
\| |x|^s p_t\|_{L^2(G)}
\leq C \sqrt{t}^{s -\frac n2}. 
$$

\item
\label{item_thm_kernelf(L)_Lpest_multintL}
For any $s\geq 0$, and any $\alpha\in \bN_0^n$, 
there exists a constant $C>0$ and $d\in \bN_0$ 
such that
for any $f:\bR\to \bC$ continuous and supported in $[0,1]$, we have
$$
\forall t\in (0,1]\qquad
\| |x|^s X^\alpha f(t\cL)\delta_{e_G}\|_{L^1(G)}
\leq C \sqrt{t}^{s-|\alpha|} 
\sup_{\substack {\lambda\geq 0\\ \ell =0,\ldots,d}}
|f^{(\ell)}(\lambda)|,
$$
and 
$$
\forall t\in (0,1]\qquad
\| |x|^s f(t\cL)\delta_{e_G}\|_{L^2(G)}
\leq C \sqrt{t}^{s-\frac n2} 
\sup_{\substack {\lambda\geq 0\\ \ell =0,\ldots,d}}
|f^{(\ell)}(\lambda)|,
$$
in the sense that if the right-hand sides above are finite, 
then the left-hand sides are also finite and  the inequalities hold.
\end{enumerate}
\end{theorem}

Above, $|x|=d(x,e_G)$ denotes the Riemannian distance on the Riemann between $x$ and the neutral element $e_G$
as it is always possible to define a left-invariant Riemannian distance on $G$, denoted by $d(\cdot,\cdot)$.
We also denote by $B(r):=\{|x|<r\}$ the ball about $e_G$ of radius $r>0$. 
The volume $V(r)$ of the ball $B(r)$.
It may be estimated via
$V(r):= |B(r)|
 \sim r^n $ for $r$ small.

\begin{proof}[Proof of Theorem \ref{thm_kernelf(L)}]
For Part (1) see  \cite{varo}.

For Part (2), with $p=1,2$, we decompose the integral $\int (|x|^{s} p_t)^p$
as $\int_{|x|\leq \sqrt t}+\sum_{j=0}^\infty \sum_{|x|\sim 2^j \sqrt t}$.
Using the estimates of Part (1) together with 
 $\int_G e^{-\frac{|x|^2}{Ct}} dx
\leq C V(\sqrt t)$, see  \cite[p.111]{varo}, 
we obtain the bound.

Part (3) was proved in \cite[Lemma Appendix A.6]{monJFA}
for the $L^1$-norms using the estimates of the heat kernel and its derivatives which can be deduced from \cite{varo}.
The methods were classical and mainly due to Alexopoulos, see \cite{alexo}.
They can be easily adapted to the case of  the $L^2$-norms.
\end{proof}

We can now prove Proposition \ref{prop_f(L)_sup}.
Part (1) follows from Proposition \ref{prop_dotHs} and
Theorem \ref{thm_kernelf(L)} Part \eqref{item_thm_kernelf(L)_Lpest_multintL}.
Let us prove Part (2).
We fix a dyadic decomposition, 
that is, a function $\eta_0\in \cD(0,\infty)$ valued in $[0,1]$ 
such that 
$$
\sum_{j\in \bZ}\eta_j=1\quad \mbox{on}\ (0,\infty), \quad
\mbox{where}\ \eta_j(\lambda):=\eta_0(2^{-j}\lambda), 
\ \lambda\in \bR, \ j\in \bZ.
$$
As the spectrum of $\cL$ is a discrete subset of $[0,\infty)$
with no accumulation point, 
there exists $j_0$ depending on $\eta_0$ and $G$, such that $\eta_j(\widehat\cL)=0$ for all $j<j_0$.
Furthermore the 0-eigenspace is $\bC 1_G$, 
where $1_G$ denotes the  function on $G$ constantly equal to 1.
Recall $\widehat 1_G= \delta_{1_{\Gh}}$.
We have:
$$
f(\widehat\cL)=f(0) \delta_{1_{\Gh}} + \sum_{j\in \bZ} (f\eta _j)(\widehat \cL).
$$
The sum over $j\in \bZ$ is in fact over $j>j_0$ and 
we have
$$
\|\delta_{1_{\Gh}} \|_{\dot H^s(\Gh)}=\|1_G\|_{L^2(q_1^s(x)dx)}
=\|q_1^s\|_{L^1(G)}<\infty.
$$
Hence we have obtained 
\begin{equation}
\label{eq_pf_prop_f(L)}
\|f(\widehat \cL)\|_{\dot H^s(\Gh)}
\leq
|f(0)|\|q_1^s\|_{L^1(G)}
+
 \sum_{j>j_0} \|(f(2^j\cdot) \eta _1)(2^{-j}\widehat \cL)\|_{\dot H^s(\Gh)}.
\end{equation}
We estimate each term in the sum using  Part (1):
\begin{align*}
\|(f\eta _j)(\widehat \cL)\|_{\dot H^s(\Gh)}
&=
\|(f(2^j\cdot) \eta _1)(2^{-j}\widehat \cL)\|_{\dot H^s(\Gh)}
\lesssim 
2^{-\frac j2 (s-\frac n2)} 
\sup_{\lambda\geq 0, \ell =0,\ldots,d}
|\partial_\lambda^{\ell}\{(f(2^j\lambda) \eta _1 (\lambda)\}|,
\\
&\lesssim 
2^{-\frac j2 (s-\frac n2)} 
\sup_{\lambda\geq 0, \ell_1 =0,\ldots,d}
(1+\lambda)^{\ell_1}
|f^{(\ell_1)}(\lambda) |.
\end{align*}
Therefore 
$$
\|f(\widehat \cL)\|_{\dot H^s(\Gh)}
\lesssim |f(0)| + \sum_{j>j_0}2^{-\frac j2 (s-\frac n2)} 
\sup_{\lambda\geq 0, \ell_1 =0,\ldots,d}
(1+\lambda)^{\ell_1}
|f^{(\ell_1)}(\lambda) |,
$$
and the conclusion follows.

\subsection{Proof of Proposition \ref{prop_f(L)}}
\label{subsec_pf_prop_f(L)}

The main technical point in the proof  of Proposition \ref{prop_f(L)} is the following estimates:
\begin{lemma}
\label{lem_prop_f(L)}
Let $G$ be a compact Lie group of dimension $n$.
\begin{enumerate}
\item 
Let $s,s'\in \bR$ with
$s'>s+\frac 12$ and $s\geq 1$.
There exists a constant $C>0$
such that for every $f\in H^s(\bR)$ 
supported in $[0,1]$
and every $t\in (0,1)$ we have
$$
\|f (t\widehat{\cL} )\|_{\dot H^{s}(\Gh)}
\leq C
t^{\frac 12 (s-\frac n2)} 
\|f\|_{H^{s'}(\bR)}.
$$
\item 
There exists a constant $C>0$
such that for every function $f:\bR\to \bC$ continuous 
and  
supported in $[0,1]$,
and every $t\in (0,1)$ we have
$$
\|f (t\widehat{\cL} )\|_{L^2(\Gh:\Sigma)}
\leq C
t^{-\frac n4} 
\sup_{\lambda\geq 0} |f(\lambda)|.
$$
\end{enumerate}
\end{lemma}

Admitting this lemma, the proof of Proposition \ref{prop_f(L)} follows easily:
\begin{proof}[Proof of Proposition \ref{prop_f(L)}]
Proceeding as in the proof of Proposition \ref{prop_f(L)_sup} Part (2),
we obtain the inequality
\eqref{eq_pf_prop_f(L)}.
Note that 
$$
|f(0)|\leq \sup_{\lambda\geq 0} |f(\lambda)|
\lesssim \sup_{r>0} \|f (r \, \cdot\, ) \eta_1\|_{H^{s_1}},
$$
for any $s_1>1/2$ by the Sobolev embeddings on $\bR$.

To estimate each term of the sum in \eqref{eq_pf_prop_f(L)}, 
we proceed as follows.
Thanks to Lemma \ref{lem_prop_f(L)},
we can interpolate 
the operator $\phi\mapsto (\phi\eta_1 )(t\cL)\delta_{e_G}$
between $H^{s'}_0(\Omega)=W^{s',2}_0(\Omega)\to L^2(q_1^s(x)dx)$
and 
$W^{\epsilon,\infty}_0(\Omega)\to L^2(dx)$, $\epsilon>0$;
here $\Omega$ denotes an open interval of $(0,\infty)$ containing the support of $\eta_1$, and the subscript 0 means that the Sobolev spaces are obtained by density of the smooth functions with compact support in $\Omega$.
This yields for any $ s>s'>0$:
$$
\exists C>0 \qquad \forall t\in (0,1)\qquad
\|(\phi\eta _1)(t\widehat \cL)\|_{\dot H^s(\Gh)}
\lesssim 
t^{\frac 12(s-\frac n2)} \|\phi \|_{W^{s',p}_0(\Omega)},
$$
where $p>2$ is such that $s\geq \theta$, $s'-s\geq \theta \epsilon$, $\epsilon>\theta/2$, 
with $\epsilon\in (0,1)$ arbitrarily small  and $\theta:=2/p$.
The Sobolev embeddings gives
$\|\phi \|_{W^{s',p}_0(\Omega)}
\lesssim 
\|\phi \|_{W^{\tilde s',2}_0(\Omega)}$
whenever $\tilde s'-s'\geq \frac 12 -\frac 1p$.
Therefore, choosing $p$ large enough, we have obtained:
$$
\forall s'>s>0\qquad
 \forall t\in (0,1)\qquad
\|(\phi\eta _1)(t\widehat \cL)\|_{\dot H^s(\Gh)}
\lesssim_{s',s} 
t^{\frac 12(s-\frac n2)} \|\phi \|_{H^{s'}_0(\Omega)}.
$$
We then apply this to $t=2^{-j}$ and $\phi=f(2^j \, \cdot) \tilde \eta_1$, 
where $\tilde \eta_1\in \cD(\Omega)$ is such that $\tilde \eta_1\equiv 1$ on the support of $\eta_1$.
We have therefore obtained for any $s'>s>0$ and any $s_1>n/2$:
$$
\|f(\widehat \cL)\|_{\dot H^s(\Gh)}
\lesssim_{s',s,s_1}  
\sup_{r>0} \|f (r \, \cdot\, ) \eta_1\|_{H^{s_1}}
+
\sum_{j>j_0} 2^{-\frac j2(s-\frac n2)} \| f(2^j \, \cdot) \tilde \eta_1\|_{H^{s'}_0(\Omega)},
$$
and the conclusion follows. 
\end{proof}

It remains to show Lemma \ref{lem_prop_f(L)}.

\begin{proof}[Proof  of Lemma \ref{lem_prop_f(L)}]
Part (2) follows 
by functional calculus and  Theorem \ref{thm_kernelf(L)} Part  \eqref{item_thm_kernelf(L)_Lpest_multintL}.
Let us prove Part (1).
Let $f\in \cD(\bR)$ supported in $[0,1]$.
We write $f(\lambda)$ as $g(e^{-\lambda})$
with $g\in H^{s'}(\bR)$ supported in $[e^{-1}, 1]$.
Decomposing $g$ into Fourier series, we have
$g(\mu)=\sum_{k\in \bZ} a_k e^{ik\mu}$ for all $\mu\in (-\pi,\pi)$.
Note that 
\begin{equation}
\label{eq_sumkak}
\sum_{k\in \bZ} |a_k|^2 (1+k^2)^{s'}
\asymp \|g\|_{H^{s'}}^2
\asymp \|f\|_{H^{s'}}^2
\end{equation}
is finite.
We see 
$f(\lambda)=\sum_{k\in \bZ} a_k e^{ike^{-\lambda}}$
for any $\lambda\geq0$,
and by functional calculus, 
\begin{align*}
\| f(t\widehat \cL)\|_{\dot H^{s}(\Gh)}
&\leq
\sum_{k\in \bZ} |a_k| \|e^{ike^{-t\widehat \cL}} \|_{\dot H^{s}(\Gh)}
\\&\leq
\left(\sum_{k\in \bZ} |a_k|^2 (1+k^2)^{s'}\right)^{\frac 12}
\left(\sum_{k\in \bZ} (1+k^2)^{-s'}
\|e^{ike^{-t\widehat \cL}} \|_{\dot H^{s}(\Gh)}^2\right)^{\frac 12},
\end{align*}
by Cauchy-Schwartz' inequality.

The conclusion follows from \eqref{eq_sumkak} and 
 Lemma \ref{lem_technical_lem_prop_f(L)} below.
\end{proof}

\begin{lemma}
\label{lem_technical_lem_prop_f(L)} 
Let $G$ be a compact Lie group of dimension $n$.
For any $s\geq 1$
there exists $C=C_{s,p}>0$ such that
$$
\forall t\in (0,1), \quad\forall k\in \bZ\qquad
\|e^{ike^{-t\widehat \cL}} \|_{\dot H^s(\Gh)}
\leq C \sqrt t^{s-\frac n2} |k|^s.
$$
\end{lemma}

\begin{proof}[Proof of Lemma \ref{lem_technical_lem_prop_f(L)}]
First let us see that Lemma \ref{lem_technical_lem_prop_f(L)} is implied by the following property: for any $s\in \bN$ and $p=2,\infty$,
there exists $C=C_{s,p}>0$ such that
\begin{equation}
\label{eq_lem_technical_lem_prop_f(L)}
\forall t\in (0,1), \quad\forall k\in \bN\qquad
\|e^{ike^{-t\widehat \cL}} \|_{\dot L^p_s(\Gh : \Sigma)}
\leq C \sqrt t^{s-\frac np} k^s.
\end{equation}
By convention, 
$1/p=0$ if $p=\infty$.
The spaces $\dot L^p_s$ have been defined in Section \ref{subsec_dotLpsinN}.

The case $p=2$ will imply  Lemma \ref{lem_technical_lem_prop_f(L)}
for $s\in \bN$ and $k\in \bN$.
By interpolation
(see Proposition \ref{prop_dotHs} \eqref{item_interpolation_prop_dotHs}), 
this will imply the cases $s\geq 1$.
The case $k=0$ in  Lemma \ref{lem_technical_lem_prop_f(L)}
 is trivial, see Example \ref{ex_Diffsigma=1}.
 The case $k<0$ follows from the case $k>0$ by complex conjugation.

Hence establishing \eqref{eq_lem_technical_lem_prop_f(L)} will indeed imply Lemma \ref{lem_technical_lem_prop_f(L)}.

\smallskip

We observe that the symbol $\sigma=e^{ie^{-t\widehat \cL}}$ satisfies the hypotheses of Lemma \ref{lem_sigmak}, see Remark \ref{rem_lem_sigmak}.
Using Part (3) of this lemma
and assuming that
the estimates in \eqref{eq_lem_technical_lem_prop_f(L)} hold
for $k=1,2,\ldots,k_0$, 
we obtain:
$$
\|e^{i(k_0+1)e^{-t\widehat \cL}} \|_{\dot L^p_s(\Gh : \Sigma)}
\leq C \sqrt t^{s-\frac np} S,
$$
where 
\begin{align*}
S&:=\sum_{\ell=0}^{s-1} \sum_{s_1+s_2=\ell} \sum_{j=0}^{k_0-1}
j^{s_1} (k_0-1-j)^{s_2}\\
&\leq
\sum_{\ell=0}^{s-1} \sum_{s_1+s_2=\ell} \sum_{j=0}^{k_0-1}
k_0^{\ell+1} \int_{j/k_0}^{(j+1)/k_0} x^{s_1}(1-x)^{s_2}dx
\\ 
&\leq 
\left(\sum_{\substack{s_1,s_2\in \bN_0\\ s_1+s_2\leq s-1}} 
 \int_0^1 x^{s_1}(1-x)^{s_2}dx\right)
\sum_{\ell=0}^{s-1} k_0^{\ell+1}
\ \lesssim_s  k_0^s;
 \end{align*}
this is  the same estimate in \eqref{eq_lem_technical_lem_prop_f(L)}
for $k_0+1$.
This will show \eqref{eq_lem_technical_lem_prop_f(L)}
once we have established this estimate for $k=1$; this is what we now do.
We want to estimate
$$
\|e^{ie^{-t\widehat \cL}} \|_{\dot L^p_s(\Gh : \Sigma)}
=
\|\sum_{\ell=0}^\infty \frac{i^\ell}{\ell!}
e^{-t\ell\widehat \cL} \|_{\dot L^p_s(\Gh : \Sigma)}
\leq
\sum_{\ell=1}^\infty \frac{1}{\ell!}
\|e^{-t\ell\widehat \cL} \|_{\dot L^p_s(\Gh : \Sigma)}.
$$
By Proposition \ref{prop_dotHs} Part (1)
and Lemma  \ref{lem_dotLinftys_L1G}, 
we have
$$
\|e^{-t\widehat\cL}\|_{\dot H^s(\Gh)}
\asymp
\|p_t\|_{L^2(|x|^{2s}dx)}
\quad\mbox{and}\quad
\|e^{-t\widehat\cL}\|_{\dot L^\infty_s(\Gh:\Sigma)}
\lesssim
\|p_t\|_{L^1(|x|^sdx)}, 
$$
so that the estimates on the heat kernel $p_t=e^{-t\cL}\delta_0$
given in Theorem \ref{thm_kernelf(L)}
yields
$$
\exists C=C_s\quad \forall t>0, \ p=2,\infty\qquad
\|e^{-t\widehat\cL}\|_{\dot L^\infty_s(\Gh:\Sigma)}
\leq C_s\sqrt t^{-\frac np} \min (1,\sqrt t^s), 
$$
with the convention that $\frac np=0$ if $p=\infty$.
We use these estimates in 
\begin{align*}
\|e^{ie^{-t\widehat \cL}} \|_{\dot L^p_s(\Gh : \Sigma)}
&\leq
\sum_{\ell=1}^\infty \frac{1}{\ell!}
\|e^{-t\ell\widehat \cL} \|_{\dot L^p_s(\Gh : \Sigma)}
=
\sum_{\ell:t\ell\leq 1}+
\sum_{\ell:t\ell> 1}
\lesssim 
\sum_{\ell:t\ell\leq 1}
\frac{1}{\ell!} 
(t\ell)^{\frac 12 (s-\frac np)}
+
\sum_{\ell:t\ell> 1}
\frac{1}{\ell!} 
(t\ell)^{-\frac n{2p}}
\\
&\lesssim 
C_1 t^{\frac 12 (s-\frac np)} 
+
t^{-\frac n{2p}}\sum_{\ell:t\ell> 1}
\frac{1}{\ell!} ,
\end{align*}
where $C_1$ is the finite constant $\sum_{\ell=0}^\infty\frac{1}{\ell!} 
\ell^{\frac 12 (s-\frac np)}$.
The last term above is a sum over $\ell>t^{-1}$ which can be viewed as the Taylor remainder of the exponential function on $\bR$ between 0 and 1
with order $[t^{-1}]$. 
Therefore it is bounded by $e^1 / [t^{-1}]! \lesssim t^{\frac 12 +\frac 1t}$ by the Stirling formula. 
This implies \eqref{eq_lem_technical_lem_prop_f(L)} for $k=1$.

This concludes the proofs of \eqref{eq_lem_technical_lem_prop_f(L)} for any $k\in \bN$, thus of Lemma \ref{lem_technical_lem_prop_f(L)}.
\end{proof}

This concludes the proof of Proposition \ref{prop_f(L)}.

\section{ Multiplier theorems}
\label{sec_mult_thm}

In this section we state and prove our multiplier theorems.
We start by giving some historical perspectives on the subject 
in Section \ref{subsec_mult_thm_Rn}.
This  leads us to defining the space $\cM_s$ of symbols  of compact Lie groups  in Section \ref{subsec_Sigmas}.
The membership in $\cM_s$ is a `H\"ormander condition'
and we show in  Section \ref{subsec_multG}
that this implies $L^p$-boundedness of the corresponding multiplier operator. 
We also discuss there  the sharpness of the result
and a Mihlin-type theorem.
The proof of the H\"ormander version is given in Section \ref{subsec_pf_thm_hormander}.
A Marcinkiewicz-type theorem is discussed in the last section of this paper, 
recovering as a particular case the historical theorem given at the beginning of the introduction.

\subsection{Historical perspectives}
\label{subsec_mult_thm_Rn}

The very first multiplier theorem  was on the one-dimensional torus $\bT$ and its statement (see theorem below)   strikingly suggests the use of discrete derivatives and differential structure on the dual of $\bT$: 
\begin{theorem}
\label{thm_marcinkiewicz}
[Marcinkiewicz multiplier theorem (1939) \cite{marc}]
If $\sigma :\bZ \to \bC$ is a sequence such that the  suprema
$$
\sup_{\ell}|\sigma(\ell)|<\infty
\qquad\mbox{and}\qquad
\sup_{j\in \bN_0} \sum_{2^j\leq \ell^2\leq 2^{j+1}} |\sigma(\ell) -\sigma(\ell+1)| 
$$
are finite, 
then the corresponding Fourier multiplier is bounded on $L^p(\bT)$, $1<p<\infty$.
\end{theorem}

The next known result on Fourier multipliers is on $\bR^n $:
\begin{theorem}[Mihlin's multiplier theorem (1956) \cite{mihlin}]
If a function $\sigma$ 
defined on $\bR^n\backslash\{0\}$ 
has at least $[n/2]+1$ continuous derivatives that satisfy 
\begin{equation}
\label{eq_mihlin_cond}
\forall \alpha\in \bN_0^n, \ |\alpha|\leq [n/2]+1,
\qquad
|\partial^\alpha\sigma (\xi)|\leq C_\alpha|\xi|^{-|\alpha|},
\end{equation}
(where $[t]$ is the integer part of $t$), 
then the Fourier multiplier operator $\Op(\sigma)$
associated with $\sigma$, 
initially defined on Schwartz functions via
\begin{equation}
\label{eq_T_sigma}
\Op(\sigma) \phi := \cF^{-1} \{\sigma \widehat \phi\},
\end{equation}
(where $\cF\phi=\widehat\phi$ denotes the Euclidean Fourier transform)
 admits a bounded extension on $L^p(\bR^n)$ for all $1<p<\infty$;
 for $p=1$, $\Op(\sigma)$ is bounded from $L^1$ to weak-$L^1$.
\end{theorem}

Shortly afterwards,  H\"ormander  gave a further generalisation and simplification of Mihlin's results:
\begin{theorem}[Hormander's multiplier theorem (1960)]
	If $\sigma$ is 
locally uniformly
in a Sobolev space $H^s(\bR^n)$ for some $s>n/2$, 
that is, requiring that the quantity
\begin{equation}
\label{eq_luHsRn}
\|\sigma\|_{l.u.H^s(\bR^n),\eta}:=
\sup_{r>0} \|\sigma (r\, \cdot)\ \eta \|_{H^s(\bR^n)}
\end{equation}
is finite for some  non-zero smooth function $\eta\in \cD(0,\infty)$, then 
the Fourier multiplier operator $\Op(\sigma)$
admits a bounded extension on $L^p(\bR^n)$ for all $1<p<\infty$;
 for $p=1$, $\Op(\sigma)$ is bounded from $L^1$ to weak-$L^1$.
\end{theorem}
Note that the proof only requires a supremum over $r=2^{j/2}$, $j\in \bZ$.

Multipliers problems have been extensively studied in various contexts.
In his textbook on singular integrals \cite{stein_sing},
E. Stein presented results and related problems in the Euclidean context.
Fourier multipliers  have been explored in completely abstract functional settings, see the 1971 textbook by Larsen \cite{larsen}. 
However in this paper we restrict our attention to Mihlin-H\"ormander type theorems.
On any compact manifold, 
general (and often sharp) results for  Fourier integral or pseudo-differential operators have been  known for nearly thirty years,
see e.g. \cite{seeger+sogge+stein,seeger+sogge}.
In the context of Lie groups, 
the focus of the problem `in the H\"ormander sense' has been on 
spectral multipliers 
in operators such as  sub-Laplacians:
see for instance
\cite{alexo,mauceri+meda,seeger+sogge,sikora+wright,
stein_topics} 
to cite only the results or methods mentioned in this paper.
The list of references (known to the author at the time of writing) on Fourier multipliers 
`in the H\"ormander sense' is much shorter:
for compact Lie groups
\cite{coifman+weiss,coifman+weiss1,coifman+weiss2,ruzhansky+wirth}
and \cite{clerc,stanton,strichartz,vretare, weiss, weiss_SU}
with an emphasis on central multipliers,
and for non-compact Lie groups
\cite{rubin,dM+M,fischer+ruzhansky}, 
see the discussion of these matters in the introduction (Section \ref{sec_intro}).

The multiplier problems continue to be an active subject of research in  non-commutative harmonic and functional analysis, see e.g. \cite{junge+mei+parcet}.

\subsubsection*{The H\"ormander condition}

It is a routine exercise to check that
the quantity in \eqref{eq_luHsRn} is finite independently of a choice of $\eta$, in the sense that if $\eta_1,\eta_2\in \cD(0,\infty)$ are non-zero, 
then $\|\sigma\|_{l.u.H^s(\bR^n),\eta_1}\asymp \|\sigma\|_{l.u.H^s(\bR^n),\eta_2}$.
The Sobolev embeddings imply  that for $s>n/2$, 
if the quantity in \eqref{eq_luHsRn} is finite then $\sigma$ is continuous and bounded with 
\begin{equation}
\label{eq_supRn_luHs}
\sup_{\bR^n} |\sigma|\lesssim \|\sigma\|_{l.u.H^s(\bR^n),\eta}.
\end{equation}
Furthermore
\begin{equation}
\label{eq_luHsRn_equiv}
s>n/2\ \Longrightarrow\
\|\sigma\|_{l.u.H^s(\bR^n),\eta}
\asymp
\sup_{\bR^n} |\sigma|+
\sup_{r>0} r^{s-n/2}
\|\sigma \ \eta(r^{-1} \cdot) \|_{\dot H^s(\bR^n)},
\end{equation}
having denoted $\dot H^s(\bR^n)$ the homogeneous Sobolev space on $\bR^n$.
H\"ormander's multipliers theorem is sharp in the range of $s$, 
even when restricted to  radial multipliers.
In fact the sharpness may be obtained by considering the imaginary power 
$\Op(|\xi|^{2i\alpha})=\Delta^{i\alpha}$
of the Laplace operator $\Delta=-\partial_1^2-\ldots-\partial_n^2$,
see \cite[p.1746]{sikora+wright} and \cite[p.51-52]{stein_sing}.

The H\"ormander condition with $s$ near enough $n/2$ implies Mihlin's condition in
\eqref{eq_mihlin_cond}
since we check easily
\begin{equation}
\label{eq_luHsinN}
s\in \bN\ \Longrightarrow\
\|\sigma\|_{l.u.H^s(\bR^n),\eta}
\lesssim \max_{\alpha\in \bN_0^n, |\alpha|\leq s} \sup_{\xi\in \bR^n} |\xi|^\alpha |\partial^\alpha \sigma(\xi)|.
\end{equation}

\medskip

In the next section, we will define a H\"ormander condition 
thanks to the equivalence \eqref{eq_luHsRn_equiv}.
The main issue is to define a relevant cut-off function $\eta$ that can be dilated.
On $\bR^n$, 
we can restrict ourselves to radial cut-offs, i.e. $\eta=\eta_1(|\cdot|^2)$, 
which can be viewed as the Fourier transform of the spectral multiplier
$\eta_1(\Delta)$ of the Laplace operator on $\bR^n$.
An analogue cut-off function as spectral multiplier in $\cL$ can be defined  in our context; it can also be  dilated.

\subsection{The space $\cM_s$}
\label{subsec_Sigmas}

The previous section leads us to define the space $\cM_s$ of symbol satisfying the following condition of H\"ormander type:

\begin{definition}
Let $s>n/2$ and $\eta\in \cD(0,\infty)$.
We denote by $\cM_{s,\eta}$ the space of bounded symbols $\sigma\in L^\infty(\Gh:\Sigma)$ such that  the following quantity is finite:
$$
\sup_{r>0}r^{s-\frac n2}
\| \sigma \ \eta (r^{-2} \widehat\cL)\|_{\dot H^s(\Gh)}.
$$

We then set
$$
\|\sigma\|_{\cM_s,\eta}
:=
\|\sigma\|_{L^\infty(\Gh:\Sigma)}+
\sup_{r>0}r^{s-\frac n2}
\| \sigma \ \eta (r^{-2} \widehat\cL)\|_{\dot H^s(\Gh)}.
$$
\end{definition}

In dealing with these spaces we will often use the fact that in the supremum above we may assume $r$ large:
\begin{lemma}
\label{lem_r0}
Let $s>n/2$ and $\eta\in \cD(0,\infty)$.
For any $r_0>0$, there exists $C>0$ depending only on $r_0,s,\eta$ and the structure of the group $G$ such that 
we have for any $\sigma\in \Sigma$:
$$
\|\sigma\|_{\cM_s,\eta}
\leq C
\left(
\|\sigma\|_{L^\infty(\Gh:\Sigma)}+
\sup_{r>r_0}r^{s-\frac n2}
\| \sigma \ \eta (r^{-2} \widehat\cL)\|_{\dot H^s(\Gh)}
\right),
$$
in the sense that if the right-hand side is finite then the left-hand side is finite and the  inequality holds.
\end{lemma}
\begin{proof}
Let $\eta\in \cD(0,\infty)$ with $\supp \, \eta \subset (a_0,b_0)\subset (0,\infty)$.
By Proposition \ref{prop_dotHs} Part \eqref{item_prop_dotHs_subset},
$$
\| \sigma \ \eta (r^{-2} \widehat\cL)\|_{\dot H^s(\Gh)}\lesssim 
\| \sigma \ \eta (r^{-2} \widehat\cL)\|_{\dot H^{[s]}(\Gh)}
=\sqrt{ 
\sum_{|\alpha|=s}
\|\Delta^\alpha  \sigma \ \eta (r^{-2} \widehat\cL)(\pi) \|_{L^2(\Gh, \Sigma_{\varphi^{\otimes \alpha}})}^2}.
$$
By Lemma \ref{lem_lambdapi} Part \eqref{item_lem_lambdapi_eq_lambdarhos},
if $r\leq r_0$
the summation over $\pi\in \Gh$
is in fact over a finite number of $\pi\in \Gh$ 
since they will satisfy $\lambda_\pi  < C + r_0^2b_0$.
Using 
$$
\|\Delta^\alpha  \sigma \ \eta (r^{-2} \widehat\cL)(\pi) \|_{L^2(\Gh, \Sigma_{\varphi^{\otimes \alpha}})}
 \leq C_{\pi,\alpha} \|\sigma\|_{L^\infty(\Gh:\Sigma)}, 
 $$
 we obtain 
$$
r\leq r_0
\Longrightarrow 
\| \sigma \ \eta (r^{-2} \widehat\cL)\|_{\dot H^{[s]}(\Gh)}
\leq 
C'_{[s],G,C,a_0,r_0} \|\sigma\|_{L^\infty(\Gh:\Sigma)}.
 $$
 The statement follows.
\end{proof}

The definition of $\cM_{s,\eta}$ is independent of any (non trivial) $\eta$ and 
$\|\sigma\|_{\cM_s,\eta}
\asymp\|\sigma\|_{\Sigma_\zeta, \zeta}$
for any two non-trivial $\eta,\zeta\in \cD(0,\infty)$
since we have:

\begin{lemma}
\label{lem_Sigmas_indepeta}
Let $s>n/2$ and let $\eta,\zeta \in \cD(0,\infty)$.
We assume $\eta\not\equiv 0$.
Then we have
$$
\|\sigma\|_{\cM_s, \zeta}
\leq C\|\sigma\|_{\cM_s,\eta},
$$
where $C=C_{s,G,\eta,\zeta}>0$ is a constant independent of $\sigma\in \Sigma$,
in the sense that if the left-hand side is finite, 
then the right-hand side is finite and the inequality holds.
\end{lemma}
\begin{proof}
Let $\eta$ and $\zeta$ as in the statement.
We may assume $\eta$ real valued
(otherwise we consider separately its real and imaginary parts).
Let $c_o>0$ such that  $2^{c_o}I$ intersects $I$
where  $I$ is an open interval inside the support of $\eta$.
For $\lambda\in \bR$ and $j\in \bZ$, we set
$$
\eta_j(\lambda) =\eta(2^{-c_o j} \lambda)
\quad\mbox{and}\quad
\alpha(\lambda):=\sum_{j\in \bZ} \eta_j^2(\lambda).
$$
We check easily that 
$\alpha\equiv 0$ on $(-\infty,0]$
and that on $(0,\infty)$, $\alpha$ is a smooth positive function which does not vanish.
Furthermore
$$
\forall \lambda\in \bR, \ j\in \bZ\quad
\alpha(2^{jc_o} \lambda)=\alpha(\lambda)\, ,
\qquad\mbox{thus}\qquad
\forall \lambda>0\quad
\sum_{j\in \bZ} \frac{\eta_j^2}\alpha (\lambda)
=
1 \, .
$$
Hence for any $\pi\in \Gh\backslash\{1_{\Gh}\}$, 
we have
$\id_{\cH_\pi}=\sum_{j\in \bZ} \frac{\eta_j^2}\alpha (r^{-2}\pi(\cL))$
 in $\sL(\cH_\pi)$.
Inserting the sum, we obtain
$$
\|\sigma \ \zeta(r^{-2} \widehat\cL)\|_{\dot H^s(\Gh)}
\leq
\sum_{j\in \bZ} 
\|\sigma \ (\frac{\eta_j^2}\alpha \zeta)
(r^{-2} \widehat\cL)\|_{\dot H^s(\Gh)}.
$$
In fact the sum is over $j$ such that 
$2^{jc_o}\supp\ \eta$ and $\supp\ \zeta$ 
 have a non-empty intersection, 
so this summation is finite and independent of $r$ and $\sigma$.
Proposition \ref{prop_dotHs} Part (1)
and Lemma \ref{lem_veryweakLeib} yield
\begin{align*}
\|\sigma \ (\frac{\eta_j^2}\alpha \zeta)
(r^{-2} \widehat\cL)\|_{\dot H^s(\Gh)}
&\lesssim 
\|\sigma \eta_j(r^{-2} \widehat\cL)\|_{\dot H^s(\Gh)}
\|(\frac{\eta_j}\alpha \zeta)
(r^{-2} \cL)\delta_{e_G}\|_{L^1(G)}
\\&\qquad+
\|\sigma \eta_j(r^{-2} \widehat\cL)\|_{L^2(\Gh:\Sigma)}
\|(\frac{\eta_j}\alpha \zeta)
(r^{-2} \cL)\delta_{e_G}\|_{L^1(|x|^sdx)}.
\end{align*}
By the Peter Weyl theorem, 
\begin{equation}
\label{eq_sigmaetaL2}
\|\sigma\ \eta_j(r^{-2} \widehat\cL)\|_{L^2(\Gh:\Sigma)}
\leq
\|\sigma\|_{L^\infty(\Gh:\Sigma)}
\|\eta_j(r^{-2} \widehat\cL)\|_{L^2(\Gh:\Sigma)}
\lesssim \|\sigma\|_{L^\infty(\Gh:\Sigma)} 
(2^{j}r^2)^{\frac n4}.
\end{equation}
by Theorem \ref{thm_kernelf(L)} Part (3), 
which  also provides estimates for the $L^1$-norms in $(\frac{\eta_j}\alpha \zeta)
(r^{-2} \cL)\delta_{e_G}$.
We obtain
$$
\|\sigma \ \zeta(r^{-2} \widehat\cL)\|_{\dot H^s(\Gh)}
\lesssim 
\|\sigma\|_{\cM_s,\eta} r^{-(s-\frac n2)}.
$$
This implies 
$\|\sigma\|_{\cM_s, \zeta}
\lesssim\|\sigma\|_{\cM_s,\eta}$
and
concludes the proof.
\end{proof}

Hence  we may write 
$$
\cM_s:=\cM_{s,\eta}
\quad\mbox{for any non-trivial} \
\eta\in \cD(0,\infty).
$$

A first criteria of membership of $\cM_s$ independent of a choice of a $\sigma\in \cD(0,\infty)$ is given by the following property:
\begin{lemma}
\label{lem_Sigmasa0b0}
For any $0<a_0<b_0$, $s>n/2$ and any $\eta\in \cD(0,\infty)$, 
there exists $C=C_{s,a_0,b_0,\eta}>0$
such that 
we have for any $\sigma\in \Sigma$:
$$
\|\sigma\|_{\cM_s,\eta}
\leq C\left(
\|\sigma\|_{L^\infty(\Gh:\Sigma)}+
\sup_{r>0}r^{s-\frac n2}
\| \sigma1_{[r^2a_0,r^2b_0]}(\widehat \cL) \|_{\dot H^s(\Gh)}\right).
$$
\end{lemma}
\begin{proof}
First let us fix  $0<a_0<b_0$ and consider $\eta\in \cD(0,\infty)$
with $\eta\equiv1$ on $[a_0,b_0]$.
Then $\sigma \ \eta(r^{-2}\widehat \cL) =\sigma_{r,a_0,b_0} \ \eta(r^{-2}\widehat \cL)$ and we have by Lemma \ref{lem_veryweakLeib}
\begin{align*}
&\|\sigma \ \eta(r^{-2}\widehat \cL)\|_{\dot H^s(\Gh)} 
=\| \sigma_{r,a_0,b_0} \ \eta(r^{-2}\widehat \cL)\|_{\dot H^s(\Gh)}
\\&\qquad\lesssim
\|\sigma_{r,a_0,b_0} \|_{\dot H^s(\Gh)} \|\eta(r^{-2}\cL)\delta_{e_G}\|_{L^1(G)}
+
\|\sigma_{r,a_0,b_0} \|_{\dot H^0(\Gh)} \|\eta(r^{-2}\cL)\delta_{e_G}\|_{L^1(|x|^s dx)}.
\end{align*}
Using Theorem \ref{thm_kernelf(L)} Part (3) to estimate the $L^1$-norms,  
and proceeding as in \eqref{eq_sigmaetaL2} for the $\dot H^0(\Gh)=L^2(\Gh:\Sigma)$-norm, we obtain:
$$
\|\sigma \ \eta(r^{-2}\widehat \cL)\|_{\dot H^s(\Gh)}
\lesssim
\|\sigma_{r,a_0,b_0} \|_{\dot H^s(\Gh)}+
\|\sigma\|_{L^\infty(\Gh:\Sigma)}r^{-s+\frac n2}.
$$
This together with Lemma \ref{lem_Sigmas_indepeta} implies the statement.
\end{proof}

\medskip

A criteria for membership in $\cM_s$ analogue to \eqref{eq_luHsinN} is the following:
\begin{lemma}
\label{lem_Sigmas}
Let $s\in \bN$, and let $\eta \in \cD(0,\infty)$.
There exists a constant $C=C_{s,G,\eta}>0$ such that 
for any multiplier $\sigma$ we have for any $r>0$
$$
\|\sigma\|_{\cM_s,\eta}
\leq C 
\max_{\substack{
\alpha\in \bN_0^f\\ |\alpha|\leq s}}
\sup_{\pi\in \Gh}
(1+\lambda_\pi)^{\frac{|\alpha|}2} 
\|\Delta^\alpha \sigma(\pi)\|_{\sL(\cH_{\varphi^{\otimes \alpha}} \otimes \cH_\pi)},
$$
in the sense that if the left-hand side is finite then 
$\eta (r^{-2} \widehat\cL)\sigma\in \dot H^s(\Gh)$ is finite for any $r>0$ and the inequality holds.
\end{lemma}

\begin{proof}
Let $s\in \bN$ with $s>n/2$.
Let $\eta\in \cD(0,\infty)$.
By Lemmata \ref{lem_r0} and \ref{lem_Sigmas_indepeta}, 
we may assume 
$\supp \ \eta = [\frac 12,2]$.

The Leibniz formula \eqref{eq_Leibniz} implies
\begin{align*}
&\sum_{|\alpha|=s}
\|\Delta^\alpha (\sigma \ \eta(r^{-2}\widehat \cL) )(\pi)\|_{HS(\cH_{\varphi^{\otimes \alpha}}\otimes  \cH_\pi)}
\\&\quad\lesssim_s \sum_{|\alpha_1|+\alpha_2|=s}
\|\Delta^{\alpha_1}\sigma(\pi)\|_{\sL(\cH_{\varphi^{\otimes {\alpha_1}}}\otimes  \cH_\pi)}
\|\Delta^{\alpha_2}\sigma(\pi)\|_{HS(\cH_{\varphi^{\otimes {\alpha_2}}}\otimes  \cH_\pi)}
\end{align*}

By Lemma \ref{lem_lambdapi} Part \eqref{item_lem_lambdapi_eq_lambdarhos}, 
 the supremum $C_0:=\sup |\lambda_\rho-\lambda_\pi| $
 over $\pi,\rho\in \Gh$ and 
 $\alpha_2\in \bN_0^f$, $|\alpha_2|\leq s$
 such that
 $\rho\subset \varphi^{\otimes \alpha_2}\otimes\pi$
is finite.
Hence we have
$$
\Delta^{\alpha_2} \eta(r^{-2}\widehat \cL) (\pi)
=
0
\ \mbox{when} \ \lambda_\pi\not\in [\frac {r^2}2-C_0, 2r^2 +C_0].
$$
We fix $r_0>0$ such that $1< \frac {r^2}2-C_0<2r^2 +C_0$ for all $r>r_0$.
Assuming that 
$$
M_{s'}:=
\max_{|\alpha_1|= s'}
\sup_{\pi\in \Gh}
(1+\lambda_\pi)^{\frac{s'}2} 
\|\Delta^{\alpha_1} \sigma(\pi)\|_{\sL(\cH_{\varphi^{\otimes \alpha_1}} \otimes \cH_\pi)},
$$
is finite for $s'=0,\ldots,s$, 
we have for every $r>r_0$:
\begin{align*}
&\sum_{|\alpha|=s}
\|\Delta^\alpha (\sigma \ \eta(r^{-2}\widehat \cL) )(\pi)\|_{HS(\cH_{\varphi^{\otimes \alpha}}\otimes  \cH_\pi)}
\lesssim_s 
\sum_{|\alpha_2| +s'=s}
M_{s'} r^{s'}
\|\Delta^{\alpha_2}\sigma(\pi)\|_{HS(\cH_{\varphi^{\otimes {\alpha_2}}}\otimes  \cH_\pi)}
\end{align*}
and
\begin{align*}
&\|\sigma \ \eta(r^{-2}\widehat \cL)\|_{\dot H^s(\Gh)}^2
=
\sum_{|\alpha|=s}
\sum_{\pi\in \Gh} d_\pi
\|\Delta^{\alpha} (\sigma \ \eta(r^{-2}\widehat \cL) )(\pi)\|_{HS(\cH_{\varphi}\otimes \ldots\otimes \cH_{\tau_{s}}\otimes \cH_\pi)}^2
\\&\qquad\lesssim_s
\max_{s'=0,\ldots,s}
M_{s'}^2 r^{2s'} \|\eta(r^{-2}\widehat \cL)\|_{\dot H^{s-s'}(\Gh)}^2.
\end{align*}
By Lemma \ref{lem_dotHs_L2} and  Theorem \ref{thm_kernelf(L)} Part (3), 
we have
\begin{align*}
\|\eta(r^{-2}\widehat \cL)\|_{\dot H^{s-s'}(\Gh)}
=
\|\eta(r^{-2}\cL)\delta_{e_G}\|_{L^2(q_1^{s-s'})}
\asymp
\|\eta(r^{-2} \cL)\delta_{e_G}\|_{L^2(|x|^{2(s-s')}dx)}
\lesssim r^{s-s'-\frac n2}.
\end{align*}
We have obtained for any $r>r_0$
$$
\|\sigma \ \eta(r^{-2}\widehat \cL)\|_{\dot H^s(\Gh)}
\lesssim_s
\max_{s'=0,\ldots,s}
M_{s'} r^{s'}r^{s-s'-\frac n2}
=
\max_{s'=0,\ldots,s}
M_{s'} r^{s-\frac n2}.
$$
By Lemma \ref{lem_r0},
this concludes the proof.
\end{proof}

The spectral multiplier in the Laplace-Beltrami operator provides examples 
of multipliers  in $\cM_s$:
\begin{proposition}
\label{prop_cq_prop_f(L)}
Let $G$ be a compact Lie group of dimension $n$.
Let $s'>s>\frac n2$.
For every $f$ locally uniformly in $H^{s'}(\bR)$, 
the spectral multiplier $f (\widehat{\cL} )$
is in $\cM_s$.
Moreover for every non-trivial $\eta,\eta_1\in \cD(0,\infty)$ there exists a constant $C>0$ independent of $f$ such that
$$
\|f (\widehat{\cL} )\|_{\cM_s,\eta}
\leq C
\|f\|_{l.u. H^{s'}(\bR), \eta_1}.
$$
\end{proposition}
The norm $\|\cdot\|_{l.u. H^{s'}(\bR), \eta}$
was defined via
\eqref{eq_luHsRn} on $\bR^n$.

\begin{proof}
The properties of the functional calculus  and \eqref{eq_supRn_luHs}
 imply
$$
\|f (\widehat{\cL} )\|_{L^\infty(\Gh:\Sigma)}\leq \|f\|_{L^\infty(\bR)}\lesssim 
\|f\|_{l.u. H^{s'}(\bR), \eta}.
$$
Let us apply Proposition \ref{prop_f(L)} with a function $\eta_1\in \cD(0,\infty)$ to the function $\lambda\mapsto f(\lambda)\eta(r^{-2}\lambda)$:
$$
\| f (\widehat{\cL} ) \ \eta (r^{-2} \widehat\cL)\|_{\dot H^s(\Gh)}
\lesssim 
\sup_{r_1>0}
\|f (r_1\ \cdot) \eta (r^{-2}r_1 \ \cdot) \eta_1\|_{H^{s'}(\bR)}.
$$
Since the mapping $\psi \mapsto \psi \chi$ is continuous on $H^{s'}(\bR)$ for any function $\chi\in \cD(\bR)$, 
we obtain easily
$$
\sup_{r>0}
\| f (\widehat{\cL} ) \ \eta (r^{-2} \widehat\cL)\|_{\dot H^s(\Gh)}
\lesssim 
\sup_{r,r_1>0}
\|f (r_1\ \cdot) \eta (r^{-2}r_1 \ \cdot) \eta_1\|_{H^{s'}(\bR)}
\lesssim 
\sup_{r>0}
\|f (r_1\ \cdot) \eta_1\|_{H^{s'}(\bR)}.
$$
Therefore, we have: 
$$
\|f (\widehat{\cL} )\|_{\cM_s,\eta}
\leq C
\|f\|_{l.u. H^{s'}(\bR), \eta_1}.
$$
We conclude with the equivalence of two norms $\|\cdot\|_{l.u. H^{s'}(\bR), \eta}$ for two non trivial functions $\eta\in \cD(0,\infty)$.
\end{proof}

\subsection{Multiplier theorems on $G$}
\label{subsec_multG}

The main result of the paper are  the analogues of both Mihlin and H\"ormander-type conditions
for the Fourier multiplier theorem.

We start with  the analogue of the H\"ormander-type condition.
This will use the space $\cM_s$ defined in Section \ref{subsec_Sigmas}: 
\begin{theorem}[Hormander-type multiplier theorem]
\label{thm_hormander}
Let $G$ be a compact Lie group of dimension $n$.
If the multiplier  $\sigma=\{\sigma(\pi),\pi\in \Gh\}$ 
is in $\cM_s$ for some $s>n/2$
then 
the Fourier multiplier operator $\Op(\sigma)$ is bounded on $L^p(G)$ for any $1<p<\infty$, and $L^1-L^{1,\infty}$ for $p=1$.
Furthermore, for $p\in (1,\infty)$,
$$
\|\Op(\sigma)\|_{\sL(L^p(G))}
\leq C_{\eta,p,G} \|\sigma\|_{\cM_{s,\eta} }
$$
for some non-zero $\eta\in \cD(0,\infty)$,
where the constant $C_{\eta,p,G}$ depends on $\eta$, $p$ and the structure of $G$
but not on $\sigma$.
We have a similar bound for $p=1$.
\end{theorem}

As in the case of Fourier multipliers on $\bR^n$,
Theorem \ref{thm_hormander} is sharp in $s$
for Fourier multipliers even when restricted to spectral multipliers in $\cL$.
Indeed  by Theorem \ref{thm_hormander} and 
Proposition \ref{prop_cq_prop_f(L)} we have
 $$
\forall s'>n/2, \ \forall \alpha\in \bR\qquad
 \|\cL^{i\alpha}\|_{\sL(L^1(G), L^{1,\infty}(G))} 
\lesssim_{s',G}
\|\lambda^{i\alpha}\|_{l.u.H^{s'}(\bR),\eta}
 \lesssim_{s',G} (1+|\alpha|)^{s'}
$$
whereas the arguments of Sikora and Wright \cite[Theorem 1]{sikora+wright} show that  
$$
\forall \alpha\in \bR\qquad
 \|\cL^{i\alpha}\|_{\sL(L^1(G), L^{1,\infty}(G))}  
\gtrsim_{G} (1+|\alpha|)^{n/2}.
$$

Note that Theorem \ref{thm_hormander} and 
Proposition \ref{prop_cq_prop_f(L)} yield 
the H\"ormander theorem for spectral multipliers of a Laplace-beltrami $\cL$ on $G$, 
but this result is a particular case of the celebrated result of Seeger and Sogge valid for a much larger class of operators on any compact manifold  \cite{seeger+sogge}.

\medskip

We now give the Mihlin-type Fourier multiplier theorem
which as in the case of $\bR^n$
can be obtained from 
the H\"ormander version:
indeed  Theorem \ref{thm_mihlin} follows readily from Theorem \ref{thm_hormander} together with 
Lemma \ref{lem_Sigmas}:

\begin{theorem}[Mihlin-type multiplier theorem]
\label{thm_mihlin}
Let $G$ be a compact Lie group of dimension $n$.
Let $\sigma\in L^\infty(\Gh,\Sigma)$. 
If there exists a constant $M\geq 0$ 
 such that 
$$
\forall \pi\in \Gh\, ,\quad
\forall \alpha \in \bN_0^f, \ |\alpha|\leq [n/2]+1
\qquad
\|\Delta^\alpha \sigma(\pi)\|_{\sL(\cH_{\varphi^{\otimes \alpha}}\otimes\cH_\pi)}
\leq M (1+\lambda_\pi)^{-\frac{|\alpha|}2} ,
$$
then  the Fourier multiplier operator $\Op(\sigma)$ is bounded on $L^p(G)$ for any $1<p<\infty$, and $L^1-L^{1,\infty}$ for $p=1$.
Furthermore, for $p\in (1,\infty)$,
$$
\|\Op(\sigma)\|_{\sL(L^p(G))}
\leq C_{p,G} \left(\|\sigma\|_{L^\infty(\Gh:\Sigma) }+ \max_{s'=1,\ldots,s}M_{s'}
\right).
$$
where the constant $C_{p,G}$ depends on $p$ and the structure of $G$
but not on $\sigma$.
We have a similar bound for $p=1$.
\end{theorem}

Note that Theorem \ref{thm_mihlin} is sharper than the main result in  \cite{ruzhansky+wirth}
of  Ruzhansky and Wirth who obtain a similar result 
but with $s'$  being an even integer strictly greater than $n/2$.
Furthermore the results of \cite{ruzhansky+wirth} applied to the case of the torus $\bT$ do not recover the original 1939 result of Marcinkiewicz
(Theorem \ref{thm_marcinkiewicz}), 
whereas it will be the case for the  Marcinkiewicz type theorem
proved   in Section \ref{subsec_marcinkiewicz}, 
as a consequence of (the proof of)
Theorem \ref{thm_hormander}.

\subsection{Proof of Theorem \ref{thm_hormander} }
\label{subsec_pf_thm_hormander}

Again as in the case of $\bR^n$, 
to show the H\"ormander version of our Fourier multiplier theorem 
(Theorem \ref{thm_hormander}),
it suffices to prove the case `$r=2^{j/2}$', 
that is:

\begin{proposition}
\label{prop_thm_hormander}
Let $G$ be a compact Lie group of dimension $n$.
Let $\sigma=\{\sigma(\pi),\pi\in \Gh\}\in L^\infty(\Gh:\Sigma)$ be a bounded multiplier. 
If there exists a constant $C_0>0$ satisfying
$$
\forall j\in \bZ \qquad
\| \sigma \ \eta (2^{-j} \widehat\cL) \|_{\dot H^s(\Gh)}
\leq C_0 2^{\frac j2 (\frac n2 -s)},
$$
for some $s>n/2$ and some non-zero $\eta\in \cD(0,\infty)$,
then 
the Fourier multiplier operator $\Op(\sigma)$ is bounded on $L^p(G)$ for any $1<p<\infty$, and $L^1-L^{1,\infty}$ for $p=1$.
Furthermore, for $p\in (1,\infty)$,
$$
\|\Op(\sigma)\|_{\sL(L^p(G))}
\leq C_{p,G} \left(\|\sigma\|_{L^\infty(\Gh:\Sigma) }+ C_0
\right).
$$
where the constant $C_{p,G}$ depends on $p$ and the structure of $G$
but not on $\sigma$.
We have a similar bound for $p=1$.
\end{proposition}

In this section, we prove Proposition \ref{prop_thm_hormander}, 
thereby proving the main result of this paper, i.e. Theorem \ref{thm_hormander}.
The method is classical, and we allow ourselves to sketch the ideas rather than dwelling on technical details.

We choose the function $\eta\in \cD(0,\infty)$ yielding a dyadic decomposition:
 $$
\supp\, \eta\subset (1,2)\quad\mbox{and}\quad
\forall \lambda>0\qquad \sum_{j\in \bZ} \eta(2^{-j}\lambda) =1,
$$
and we set $\eta_j(\lambda):=\eta(2^{-j}\lambda)$. 
Note that since the spectrum of $\cL$ is a discrete subset of $[0,\infty)$
with no accumulation point, 
there exists $j_0$ depending on $\eta$ and $G$, such that $\eta_j(\widehat\cL)=0$
for all $j<j_0$.

Note that if $\sigma=\delta_{\pi=1_{\Gh}}$ in the sense that 
$\sigma(\pi)=0$ unless $\pi=1_{\Gh}$ and $\sigma(1_{\Gh})=1$, 
then the inversion formula yields
$\Op(\delta_{\pi=1_{\Gh}})\phi\equiv\int_G \phi$
for any  $\phi\in \cD(G)$
so that $\|\Op(\delta_{\pi=1_{\Gh}})\|_{\sL(L^p(G))} \leq 1$ for any $p\in [1,\infty)$.
Therefore it suffices to prove Proposition \ref{prop_thm_hormander} 
for $\sigma\in L^\infty(\Gh:\Sigma)$ 
satisfying $\sigma(1_{\Gh})=0$.

Let $\sigma\in L^\infty(\Gh:\Sigma)$ such that $\sigma(1_{\Gh})=0$.
We assume that there exist $s>n/2$ and $C_0>0$ such that
$$
\forall j\in \bZ\qquad \|\sigma_j\|_{\dot H^s(\Gh)}
\leq C_0 2^{-\frac j2 (s-\frac n2)},
$$
where 
$\sigma_j :=  \sigma \ \eta_j(\widehat\cL)$.
The sum $\sum_j \Op(\sigma_j)$ converges to $\Op(\sigma)$ 
in the strong operator topology of $\sL(L^2(G))$.

Let $\kappa, \kappa_j\in \cD'(G)$ be such that 
$\sigma=\widehat\kappa$
and 
$\sigma_j=\widehat\kappa_j$.
The sum $\sum_j \kappa_j$ converges to $\kappa$ in $\cD'(G)$.
Since $\kappa_j\in L^2_{fin}(\Gh:\Sigma)$ is smooth,
by hypothesis and Proposition \ref{prop_dotHs} Part (1),  $\kappa_j\in L^2(|x|^{2s}dx)$ with
\begin{equation}
\label{eq_kappaj_hyp}
\|\kappa_j\|_{L^2(|x|^{2s}dx)}
\asymp\|\sigma_j\|_{\dot H^s(\Gh)}
\leq C_0 2^{\frac j2 (\frac n2 -s)}.
\end{equation}
So $\sum_j \|\kappa_j\|_{L^2(|x|^{2s}dx)}<\infty$
and $\kappa$ coincides with a locally integrable function on $G\backslash\{e_G\}$ in  $L^2(|x|^{2s}dx)$, 
see Proposition \ref{prop_dotHs} Part (1).
The sum $\sum_j \kappa_j$ converges to $\kappa\in L^2(|x|^{2s}dx)$ in  $L^2(|x|^{2s}dx)$.

By the fundamental theorem of singular integral
\cite[Ch. III]{coifman+weiss},
it suffices to show that 
$$
\int_{|y^{-1}y'|\lesssim |y^{-1}x| }
|\kappa(y^{-1}x) - \kappa(y^{-1}x)| dx
$$
is finite independently of $y\not=y'\in G$, and similarly for $\kappa^*: x\mapsto \bar \kappa(z^{-1})$.
We will give the proof for $\kappa$ only, 
the proof for $\kappa^*$ being similar.

It suffices to show that the sum
$$
\sum_{j=j_0}^\infty I_j(h),
\quad\mbox{where}\quad
I_j(h):=
\int_{|h|\lesssim |z| }
|\kappa_j(z) - \kappa_j(hz)| dz,
$$
is bounded independently of $h\in G\backslash\{e_G\}$.

We see that
$$
I_j(h)
\leq 
2\int_{|h|\lesssim |z| }
|\kappa_j(z)| dz
\leq 2\|\kappa_j\|_{L^2(|x|^{2s}dx)}
\|1_{|h|\lesssim |z| }\|_{L^2(|z|^{-2s}dz)},
$$
by the Cauchy-Schwartz inequality.
We compute easily
$$
\|1_{|h|\lesssim |z| }\|_{L^2(|z|^{-2s}dz)}^2
\asymp
\int_{r\sim |h|}^1 r^{-2s+ n-1}dr
\asymp 1+|h|^{- 2s+n}.
$$
Using \eqref{eq_kappaj_hyp}, we obtain:
$$
I_j(h)
\lesssim C_0 2^{-\frac j2 (s-\frac n2)}(1+|h|^{- s+\frac n 2}),  
\quad\mbox{and}\quad 
\sum_{j:|h|^{-1}\lesssim 2^{j/2}}
\int_{|h|\lesssim |z| }
|\kappa_j(z) - \kappa_j(hz)| dz
\lesssim C_0, 
$$
independently of $h$.

Hence the main problem is the sum  for $j\geq j_0$ such that $2^{j/2} \lesssim |h|^{-1}$. We may also assume that $|h|$ is small, 
and thus that $h$ is in a fixed chart of $e_G$ provided by the exponential mapping. 
We have
$$
I_j(h) \leq 
\sum_{\ell=0}^m  I_{j,\ell}(h), 
\quad\mbox{where}\quad
I_{j,\ell}(h):=
\int_{|z|\asymp 2^{-\ell}}
|\kappa_j(z) -\kappa_j(hz)| dz,
$$
and $ m\in \bN_0$ is such that
$|h|\asymp 2^{-m}$.

The  estimate for the Taylor reminder of order 1 then yields:
$$
I_{j,\ell}(h)\lesssim
|h| 
\int_{|z|\asymp 2^{-\ell}}
|\nabla \kappa_j(z)| dz
\quad\mbox{if}\ \ell\lesssim m.
$$
The Cauchy-Schwartz inequality implies
$$
\int_{|z|\asymp 2^{-\ell}}
|\nabla \kappa_j(z)| dz
\leq
\|\nabla \kappa_j\|_{L^2(G)}
\|1_{|z|\asymp 2^{-\ell}}\|_{L^2(G)}
$$
We compute $\|1_{|z|\asymp 2^{-\ell}}\|_{L^2(G)}\asymp 2^{-\ell n/2}$.
The functional calculus of $\cL$ yields
$\|\nabla \kappa_j\|_{L^2(G)}\asymp 2^{j / 2} \|\kappa_j\|_{L^2(G)}$
and we have
\begin{equation}
\label{eq_L2kappaj}
\| \kappa_j\|_{L^2(G)}=\|\sigma_j\|_{L^2(\Gh:\Sigma)}
\leq \|\sigma\|_{L^\infty(\Gh:\Sigma)} \|\eta_j(\widehat \cL)\|_{\dot H^0(\Gh)}
\lesssim \|\sigma\|_{L^\infty(\Gh:\Sigma)}  2^{j \frac n 4},
\end{equation}
by Proposition \ref{prop_f(L)_sup}.
So we have obtained:
$$
I_{j,\ell}(h)\lesssim |h| 2^{\frac j2} \|\sigma\|_{L^\infty(\Gh:\Sigma)}2^{j \frac n4}
2^{-\ell \frac n 2}
$$
and the sum $\sum I_{j,\ell}(h)$ over $j, \ell\in \bZ$
satisfying $0\leq \ell \leq m$ and $j_0\leq  j\leq 2\ell$ is finite independently of $|h|$.

To sum over $\ell<j/2$, we go back to
$$
\int_{|z|\asymp 2^{-\ell}}
|\nabla \kappa_j(z)| dz
\leq
\| |\cdot|^s |\nabla \kappa_j|\|_{L^2(G)}
\||\cdot|^{-s}1_{|z|\asymp 2^{-\ell}}\|_{L^2(G)}.
$$
We compute $\||\cdot|^{-s}1_{|z|\asymp 2^{-\ell}}\|_{L^2(G)}\asymp 2^{\ell(s- n/2)}$.
We fix $\tilde \eta\in \cD(0,\infty)$ such that $\tilde \eta \equiv 1$ on $\supp\ \eta$. We set $\tilde \eta_j(\lambda) =\tilde \eta(2^{-j}\lambda)$.
We have $\kappa_j = \tilde \eta_j (\cL) \kappa_j=\kappa_j * \tilde \eta_j (\cL)\delta_{e_G}$
and  Lemma \ref{lem_veryweakLeib} yields:
$$
\| |\cdot|^s |\nabla \kappa_j|\|_{L^2(G)}
\lesssim
\||\cdot|^s|\kappa_j|\|_{L^2(G)} 
\||\nabla\tilde \eta_j (\cL)\delta_{e_G}|\|_{L^1(G)}
 + 
\|\kappa_j\|_{L^2(G)}  
\||\cdot|^s |\nabla\tilde \eta_j (\cL)\delta_{e_G}|)\|_{L^1(G)}.
$$
Using \eqref{eq_kappaj_hyp} and \eqref{eq_L2kappaj} 
for the $L^2$-norm in $\kappa_j$, 
and Theorem \ref{thm_kernelf(L)} Part   \eqref{item_thm_kernelf(L)_Lpest_multintL}
 for the $L^1$-norm with $\eta$ yields:
$$
\| |\cdot|^s |\nabla \kappa_j|\|_{L^2(G)}
\lesssim 
C_0 2^{\frac j2 (\frac n2-s)} (2^{-j})^{-\frac 12}
+ \|\sigma\|_{L^\infty(\Gh:\Sigma)} 2^{jn/4} (2^{-j})^{\frac{s- 1}2}
=(C_0+ \|\sigma\|_{L^\infty(\Gh:\Sigma)})2^{\frac j2 (\frac n2-s+1)}.
$$
Hence we have obtained
$$
I_{j,\ell}(h)\lesssim |h| (C_0+ \|\sigma\|_{L^\infty(\Gh:\Sigma)})2^{\frac j2 (\frac n2-s+1)}2^{\ell(s- \frac n 2)},
$$
and the sum $\sum I_{j,\ell}(h)$ over $j, \ell\in \bZ$
satisfying $0\leq \ell \leq j/2$ and $j_0\leq  j\leq 2m$ is finite independently of $h$.

This concludes the proof of Proposition \ref{prop_thm_hormander}, 
and of Theorem \ref{thm_hormander}.

\subsection{A Marcinkiewicz type condition}
\label{subsec_marcinkiewicz}

Our main result Theorem \ref{thm_hormander}, that is, the multiplier theorem of H\"ormander type, 
is in term of the space $\cM_s$
and
 can be viewed as a decay of $L^2$-weighted norms of the kernel localised in frequency.
As a corollary, we obtained a Mihlin version in terms of decay of derivatives of the symbol, see Theorem \ref{thm_mihlin}.
We now present a Marcinkiewicz-type theorem 
in the sense that the condition is expressed with the partial sum of derivatives being finite.

We will need the following notation: 
if $s_0\in \bN$ and $\sigma\in \Sigma$, we set
$$
\|\sigma\|_{\dot L^1_{s_0}(\Gh : \Sigma)}
:=
\sum_{|\alpha|=s_0}
\sum_{\pi\in \Gh} d_\pi
\tr_{\cH_{\varphi^{\otimes \alpha}} \otimes  \cH_\pi}
|\Delta^\alpha \sigma(\pi)|.
$$
This quantity may be finite or infinite.
\begin{remark}
\label{rem_lem_trace}
Because of the equivalence of matrix norms, 
we can replace the trace above with any norm of 
the finite dimensional space of endomorphism in $\cH_{\varphi^{\otimes \alpha}}$ valued in the matrix space over $\cH_\pi$ equipped with the Schatten norm
$$
\|M\|_{S_1(\cH_\pi)}:=\tr |M| = \tr (M^*M)^{\frac12}.
$$
 The constants involved in this equivalence depend on the spaces $\cH_{\varphi}$, $\varphi\in \FundG$, so on the structure of $G$, 
and not on $\pi\in \Gh$.
So for instance we have:
\begin{align*}
\|\sigma\|_{\dot L^1_{s_0}(\Gh : \Sigma)}
&\asymp
\sum_{|\alpha|=s_0}
\sum_{\pi\in \Gh} d_\pi
\|\Delta^\alpha \sigma(\pi)\|_{\sL(\cH_{\varphi^{\otimes \alpha}} \otimes, S_1( \cH_\pi))}
\\&\asymp
\sum_{|\alpha|=s_0}
\sum_{\pi\in \Gh} d_\pi
\sum_{i_1,j_1,\ldots,i_{s_0},j_{s_0}}
\|[\Delta^\alpha\sigma(\pi)]_{i_1,,j_1,\ldots,i_{s_0},j_{s_0}}\|_{S_1( \cH_\pi)}.
\end{align*}
\end{remark}

This quantity enables us to obtain the following $L^\infty$-weighted estimates:
\begin{lemma}
\label{lem_trace}
For any $s_0\in \bN$, 
there exists a constant $C=C_{G,s_0}>0$
such that we have for all $\phi\in \cD(G)$:
$$
\max_{x\in G} |x|^{s_0} |\phi(x)|
\leq C \|\widehat \phi\|_{\dot L^1_{s_0}(\Gh : \Sigma)}.
$$
\end{lemma}

\begin{proof}[Proof of Lemma \ref{lem_trace}]
Let $\phi\in \cD(G)$.
The Fourier inversion formula 
 and the definition of the difference operators (see Section \ref{subsec_my_diff_op})
   imply
$$
[\varphi(x)-\id]_{i,j}\phi(x)=\sum_{\pi\in \Gh} d_\pi 
\tr \left[\Delta_\varphi\widehat \phi(\pi)\right]_{i,j},
\quad\mbox{so}\quad
|[\varphi(x)-\id]_{i,j}\phi(x)|\leq
\sum_{\pi\in \Gh} d_\pi 
\tr \left|\left[\Delta_\varphi\widehat \phi(\pi)\right]_{i,j}\right|.
$$
By Proposition \ref{prop_q1} Part (1), we have
$$
\max_{x\in G} |x||\phi(x)|
\asymp
\max_{G} \sqrt q_1 |\phi|
\asymp
\sum_{\varphi\in \FundG}
\sum_{i,j} 
\max_{x\in G} |[\varphi(x)-\id]_{i,j}\phi(x)|,
$$
and the computation above then yields:
$$
\max_{x\in G} |x||\phi(x)|
\lesssim 
\sum_{\varphi\in \FundG}
\sum_{\pi\in \Gh} d_\pi 
\sum_{i,j} 
\tr \left|\left[\Delta_\varphi\widehat \phi(\pi)\right]_{i,j}\right|
$$
By Remark \ref{rem_lem_trace}, this implies the result for $s_0=1$. 
The proof for $s_0>1$ is similar and is left to the reader.
\end{proof}

The $L^\infty$-weighted estimates in Lemma \ref{lem_trace}
easily yields some $L^2$-weighted estimates of localised kernels,
and we obtain: 

\begin{corollary}
\label{cor_lem_trace}
\begin{enumerate}
\item 
For any $\eta\in \cD(0,\infty)$, $s>0$ and any $s_0\in \bN$, 
there exists a constant $C=C_{s,G,s,s_0,\eta}>0$ such that
for any $\sigma\in \Sigma$ and $r\geq 1$, we have
$$
\|\sigma\, \eta(r^{-2} \widehat\cL)\|_{\dot H^s(\Gh)}
\leq C 
\left(
\|\sigma\, \eta(r^{-2} \widehat\cL)\|_{\dot L^1_{s_0}(\Gh : \Sigma)}
r^{-s+s_0-\frac n2}
+ \|\sigma\|_{L^\infty(\Gh:\Sigma)} r^{\frac n2 -s}\right), 
$$
in the sense that if 
$\sigma\|_{L^\infty(\Gh:\Sigma)}$ and  $\|\sigma\|_{\dot L^1_{s_0}(\Gh : \Sigma)}$ are finite, then the inequality holds.
\item 
If $\sigma\in \Sigma$ is such that 
$\sigma\|_{L^\infty(\Gh:\Sigma)}$ and  $\|\sigma\|_{\dot L^1_{s_0}(\Gh : \Sigma)}$
are finite with $s_0\in \bN$, $s_0\leq n$, 
then $\sigma\in \Sigma_{s}$ for any $s>\frac n2$.
Moreover 
we have
$$
\|\sigma\|_{\cM_{s,\eta}}\leq C \left(
 \|\sigma\|_{L^\infty(\Gh:\Sigma)}
 +
 \sup_{r\geq 1}\|\sigma\, \eta(r^{-2} \widehat\cL)\|_{\dot L^1_{s_0}(\Gh : \Sigma)}
\right), 
$$
where $\eta\in \cD(0,\infty)$ is a fixed non-trivial function, 
and  $C=C_{s,G,s,s_0,\eta}>0$ is a constant, both independent of $\sigma$.
\end{enumerate}
\end{corollary}
\begin{proof}
Let $\sigma\in \Sigma$ and $\eta\in \cD(0,\infty)$.
We set $\kappa_r:=\cF_G^{-1}\sigma\, \eta(r^{-2} \widehat\cL)\in \cD(G)$ for any $r>0$.
By Proposition \ref{prop_dotHs} Part (1),
$$
\|\sigma\, \eta(r^{-2} \widehat\cL)\|_{\dot H^s(\Gh)}
\asymp
\|\kappa_r \|_{L^2(|x|^{2s}dx)}
\leq
\|\kappa_r 1_{|x|\leq r^{-1}} \|_{L^2(|x|^{2s}dx)}
+
\|\kappa_r 1_{|x|> r^{-1}} \|_{L^2(|x|^{2s}dx)}.
$$
By Lemma \ref{lem_trace}, 
we have:
$$
\||\cdot |^{s_0}\kappa_r \|_\infty\lesssim \|\sigma\, \eta(r^{-2} \widehat\cL)\|_{\dot L^1_{s_0}(\Gh : \Sigma)},
$$
so for any $r\geq 1$
$$
 \|\kappa_r 1_{|x|\leq r^{-1}} \|_{L^2(|x|^{2s}dx)}
 \lesssim
 \|\sigma\, \eta(r^{-2} \widehat\cL)\|_{\dot L^1_{s_0}(\Gh : \Sigma)}
\sqrt{\int_{|x|\leq r^{-1}} |x|^{2s-2s_0} dx}
\lesssim 
 \|\sigma\, \eta(r^{-2} \widehat\cL)\|_{\dot L^1_{s_0}(\Gh : \Sigma)}
 r^{-(s-s_0+\frac n2)}.
$$
For the second $L^2$-norm, we have
$$
\|\kappa_r 1_{|x|> r^{-1}} \|_{L^2(|x|^{2s}dx)}
\leq 
r^{-s}
\|\kappa_r \|_{L^2(G)},
$$
and by functional calculus  and  Theorem \ref{thm_kernelf(L)} Part \eqref{item_thm_kernelf(L)_Lpest_multintL},
$$
\|\kappa_r \|_{L^2(G)}\leq \|\sigma\|_{L^\infty(\Gh:\Sigma)}
\|\eta(r^{-2}\cL)\delta_{e_G}\|_{L^2(G)}
\lesssim \|\sigma\|_{L^\infty(\Gh:\Sigma)}
r^{\frac n2}.
$$
Collecting the various estimates proves Part (1).

Part (2) follows easily from Part (1) and the properties of $\cM_s$, see Section \ref{subsec_Sigmas}.
\end{proof}

Corollary \ref{cor_lem_trace} Part (2) gives a sufficient condition for the membership of the symbol in $\cM_s$, $s>n/2$, 
therefore a sufficient condition for $L^p$-boundedness of the corresponding operator using H\"ormander-type hypotheses in Theorem \ref{thm_hormander}.

Here we choose to combine Proposition \ref{prop_thm_hormander} 
with Corollary \ref{cor_lem_trace} Part (1) 
to obtain a sufficient condition `with $r=2^{j/2}$', 
which we view as a Marcinkiewicz-type property:

\begin{theorem}[Marcinkiewicz-type multiplier theorem]
\label{thm_marc-type}
Let $G$ be a compact Lie group of dimension $n$.
If the symbol  $\sigma=\{\sigma(\pi),\pi\in \Gh\}$ 
is such that
$\sigma\in L^\infty(\Gh:\Sigma)$
and 
$$
\exists s_0\in \bN, s_0\leq n\quad
\exists C_0>0\qquad 
\forall j\geq 0\qquad
\|\sigma\, 1_{[2^{j}, 2^{j+1}]} (\widehat\cL)\|_{\dot L^1_{s_0}(\Gh : \Sigma)}
\leq C_0 2^{\frac j2(n-s_0)}
$$
then 
the Fourier multiplier operator $\Op(\sigma)$ is bounded on $L^p(G)$ for any $1<p<\infty$, and $L^1-L^{1,\infty}$ for $p=1$.
Furthermore, for $p\in (1,\infty)$,
$$
\|\Op(\sigma)\|_{\sL(L^p(G))}
\leq C_{s_0,p,G}
\left( C_0+\|\sigma\|_{L^\infty(\Gh:\Sigma)}\right)
$$
where the constant $C_{s_0,p,G}$ depends on $s_0$, $p$ and the structure of $G$
but not on $\sigma$.
We have a similar bound for $p=1$.
\end{theorem}

\begin{proof}
We fix  $s>n/2$, $\eta\in \cD(0,\infty)$ and $j_0\in \bZ$ as in the proof of Proposition \ref{prop_thm_hormander}.
In particular, 
$ \eta (2^{-j} \widehat\cL) =0$ for $j<j_0$.
Proceeding as in the proof of Lemma \ref{lem_Sigmasa0b0}, 
we have for any $j\in \bZ$: 
$$
\| \sigma \ \eta (2^{-j} \widehat\cL) \|_{\dot H^s(\Gh)}
\lesssim 
\| \sigma \ 1_{[2^{j}, 2^{j+1}]} (\widehat\cL) \|_{\dot H^s(\Gh)}
+ 
2^{-\frac j2 (s-\frac n2)} \|\sigma\|_{L^\infty(\Gh:\Sigma)},
$$
whereas proceeding as in the proof of Lemma \ref{lem_r0}, 
we have for any   $j\in \bZ$, $j<0$, 
$$
\| \sigma \ \eta (2^{-j} \widehat\cL) \|_{\dot H^s(\Gh)}
\lesssim 
 \|\sigma\|_{L^\infty(\Gh:\Sigma)}.
$$

By Corollary \ref{cor_lem_trace} Part (1), 
we have for any $j\in \bN_0$
$$
\| \sigma \ 1_{[2^{j}, 2^{j+1}]} (\widehat\cL) \|_{\dot H^s(\Gh)}
\lesssim 
\|\sigma\, 1_{[2^{j}, 2^{j+1}]} (\widehat\cL) \|_{\dot L^1_{s_0}(\Gh : \Sigma)}
2^{\frac j2 (-s+s_0-\frac n2)}
+ \|\sigma\|_{L^\infty(\Gh:\Sigma)} 2^{\frac j2(\frac n2 -s)}.
$$

Combining all the estimates above together with the hypothesis of the statement implies that $\sigma$ satisfies the hypotheses of Proposition \ref{prop_thm_hormander} and the conclusion follows.
\end{proof}

Applying Theorem \ref{thm_marc-type} to the one-dimensional torus $G=\bT$ with $s_0=n=1$, we recover the historical 1939 theorem 
stated in Theorem \ref{thm_marcinkiewicz}.

\end{document}